\documentclass[preprint,12pt,square,comma,sort&compress]{elsarticle}

\usepackage{makecell}   
\usepackage{array}      
\usepackage{geometry}   
\usepackage{bm}
\usepackage[utf8]{inputenc}
\usepackage{indentfirst}
\usepackage{enumitem}
\usepackage{mathrsfs}
\usepackage{dsfont}
\usepackage{amssymb}
\usepackage{stmaryrd}
\usepackage{cite}
\usepackage{amsfonts}
\usepackage{graphicx}
\usepackage{multirow}
\usepackage{amsmath}
\usepackage{pifont}
\usepackage[colorlinks=true]{hyperref}  
\usepackage{cleveref}                    
\usepackage{color}
\usepackage{amsthm} 
\allowdisplaybreaks[4]
\usepackage{caption}

\def\sgn{\mathop{\rm sgn}}
\def\ds{\displaystyle}

\newtheorem{lemma}{Lemma}[section]

\newtheorem{theorem}{Theorem}[section]

\begin{document}

\begin{frontmatter}

\title{Optimal Harvesting of a Stochastic Logistic Model Driven by One-Sided Tempered Stable Process \tnoteref{myr}}
\tnotetext[myr]{The work was supported by the National Natural Science Foundation of China (No. 12071292, 42450269).}

\author{Wenmin Deng}
\author{Fu Zhang\corref{cor1}}
\address{College of Science, University of Shanghai for Science and Technology, Shanghai 200093, China}

\cortext[cor1]{Corresponding author.}
\ead{fuzhang82@gmail.com}

\begin{abstract}

This paper investigates a class of stochastic Logistic harvesting models driven by tempered stable processes, with a one-sided power-law L\'evy measure. We establish threshold conditions for population extinction and persistence, prove the distributional stability of the model, and derive explicit solutions for the optimal harvesting effort and the maximum sustainable yield. We systematically analyze the effects of white noise intensity and L\'evy jump intensity on the optimal harvesting strategy. In particular, by focusing on the intrinsic structural parameters of the L\'evy measure, namely the stability index and the tempering parameter, we elucidate their roles in shaping the optimal strategy and propose four targeted intervention strategies. Numerical simulations are presented to validate the theoretical findings.

\end{abstract}


\end{frontmatter}

\thispagestyle{empty}

\rm
\section{Introduction}
\label{1}
Population dynamic modeling constitutes a core foundation of resource management, ecological conservation, and sustainable utilization. Owing to its simplicity and biological interpretability, the classical Logistic model has long been widely used to describe single-species population growth under environmental carrying capacity constraints, and has been further extended to the study of optimal harvesting strategies for renewable resources such as fisheries and forestry \citep{y1,y2,y3,y4,y5}.

However, traditional deterministic models or stochastic models driven solely by Gaussian noise often neglect the impact of sudden environmental perturbations, such as extreme climate events, disease outbreaks, pollution accidents, or abrupt policy changes. Although these disturbances may occur infrequently, they can exert significant shocks on population size and may even trigger irreversible collapse. More importantly, in modern ecological systems, anthropogenic interventions have become drivers as important as, or even more dominant than, natural disturbances. For example, billions of fish are released annually worldwide to restore depleted fisheries; endangered species are replenished through habitat protection and reintroduction programs; and policy instruments such as marine protected areas and fishing moratoria directly alter population growth trajectories. These interventions are neither continuous and smooth nor sparse and isolated, but instead exhibit complex features across multiple scales and asymmetries, predominantly characterized by positive inputs. They include both routine small-scale supplementation and sporadic large-scale actions. Moreover, from a modeling perspective, the timing and magnitude of these interventions can be reasonably regarded as random variables, even though they may be somewhat planned in actual operation. This is because management practice involves many uncontrollable factors (such as implementation delays, environmental constraints, and budget fluctuations), which make the actual outcome of interventions inherently random. How to mathematically unify and characterize such a spectrum of positive perturbation mechanisms, and to assess their impacts on resource management strategies, has become a key problem in mathematical ecology.

Tempered stable processes provide an effective framework for modeling such a spectrum of positive perturbation mechanisms. In particular, stochastic differential equations driven by L\'evy jump processes offer powerful tools for describing non-Gaussian, discontinuous environmental noise. Among various jump models, tempered stable processes have attracted considerable attention due to their dual characteristics: on one hand, their power-law kernel \( z^{-1-\beta} \) captures the aggregation and heavy-tailed behavior of small and moderate jumps; on the other hand, the exponential tempering term \( e^{-\lambda z} \) suppresses the probability of extreme large jumps, ensuring finite moments of all orders and avoiding the theoretical and practical difficulties caused by infinite variance in pure \(\alpha\)-stable processes. This ``heavy-tailed but controllable'' structure enables tempered stable processes to capture frequent small-to-moderate perturbations observed in real systems while excluding physically unrealistic catastrophic jumps, making them more consistent with empirical observations in physical and biological systems.

This modeling advantage has established tempered stable processes as key tools across multiple disciplines. In physics, tempered stable processes (often referred to as tempered L\'evy flights) are used to model anomalous diffusion in turbulent flows, capturing both long-range jump behavior and finite variance at large time scales, thereby reconciling with classical diffusion theory \citep{y7,y8}. In finance, tempered stable processes, such as the CGMY model, allow flexible parameter control to accurately fit the ``sharp peak and heavy tail'' features of asset returns, while ensuring mathematical rigor in risk-neutral pricing due to the existence of exponential moments \citep{y9,y10}. In seismology, tempered power-law distributions are introduced to impose upper bounds on earthquake energy, leading to hazard models more consistent with geophysical reality \citep{y11}. Notably, the core idea of tempering, exponentially truncating heavy tails to balance extreme events and system stability, has also received empirical support in ecological and epidemiological contexts. For instance, \citep{y12} analyzed human mobility data and found that short-time displacement distributions exhibit ``suppressed heavy tails'' statistically consistent with tempered stable processes, providing mechanistic insights into superdiffusive disease spread under spatial constraints. Despite these advances, to the best of our knowledge, this idea has not yet been concretely formulated as a tempered stable L\'evy process with explicit analytical structure and applied to optimal control problems in population dynamics.

Although previous studies have investigated optimal control problems for stochastic population models with jumps, most restrict the jump component to finite activity L\'evy processes, particularly compound Poisson processes (see, e.g., \citep{y13,y14,y15}).
These models typically assume bounded jump amplitudes and a L\'evy measure supported on a measurable subset $\mathbb{Y} \subset (0,\infty)$ satisfying $\nu(\mathbb{Y}) < \infty$.
Such finite activity models essentially characterize disturbances as sparse and discrete events, with relatively long intervals between successive jumps. They are therefore well suited for describing rare natural disasters or abrupt policy shocks. However, they are less effective at capturing continuous, high-frequency human management activities, which typically involve a large number of small but cumulatively significant interventions.

In stark contrast, the innovation of this paper lies in introducing jump terms with different statistical characteristics. Specifically, we consider a special type of L\'evy process with infinite activity, a one-sided tempered stable process, where its L\'evy measure \( \nu \) does not satisfy \( \nu(\mathbb{Y}) < \infty \), but instead follows a well-defined one-sided power-law form (see the definition by \citep{y16}):
\begin{equation}\label{eq:levy-measure}
\nu(dz) = \begin{cases}  \dfrac{e^{-\lambda z}}{z^{1+\beta}} \, \mathrm{d}z, & \text{if } z > 0, \\ 0, & \text{if } z \leq 0, \end{cases}
\end{equation}
where \( \beta\in(0,2) \) is the stability index and \( \lambda \) is the tempering parameter. Since the model (\ref{f01}) introduces \( \sigma \) as the overall intensity coefficient of the jump disturbance, we clearly separate the parameters characterizing disturbance properties into two parts: the shape parameters \( \beta \) and \( \lambda \), which control the statistical structure, and the scale parameter \( \sigma \), which controls the overall intensity.

In the ecological context, this one-sided structure with only positive jumps has unique biological interpretability. It is well-suited to describe intervention events that cause non-negative, sudden increases in population size, and its ``infinite activity'' feature allows it to capture disturbances across different scales. More importantly, the parameters \( \beta \) and \( \lambda \) have clear managerial interpretations, which can be intuitively understood through practices such as ``prolific restocking'':
\begin{itemize}
    \item The parameter \( \beta \) mainly controls the density of small jumps and characterizes the ``implementation rate'' of restocking activities. When \( \beta \) is large, the system corresponds to a ``high-frequency, small-scale'' management paradigm (e.g., daily automated releases of a small number of fry, forming a ``drip infusion'' type of maintenance). When \( \beta \) is small, the system corresponds to a ``low-frequency, relatively large-scale'' management paradigm (e.g., large-scale restocking only once or twice in the breeding season, forming a ``pulse'' type of supplementation).
    \item The parameter \( \lambda \) mainly controls the suppression of large jumps and characterizes the ``scale constraint'' of a single restocking action. When \( \lambda \) is large, the system is in a ``scale-constrained'' state (e.g., restricted by hatchery capacity, regulations limit the amount of restocking per action); when \( \lambda \) is small, the system has the flexibility to ``large-scale emergency-ready'' (e.g., in the case of a population collapse, emergency restocking beyond normal capacity).
\end{itemize}

Compared to previous studies that restrict jump terms to L\'evy processes with finite activity \( (\nu(\mathbb{Y}) < \infty) \), the contribution of this paper lies in twofold deepening: at the modeling level, it introduces a one-sided tempered stable process with infinite activity \( (\nu(\mathbb{Y}) = \infty) \), thereby providing a unified framework in population dynamics to model a spectrum of interventions ranging from high-frequency small-scale to high-frequency large-scale; at the analytical level, it shifts the focus from the ``overall intensity'' of jumps to their ``intrinsic statistical structure''. Specifically, existing work usually treats L\'evy noise as a general disturbance term without explicitly specifying the form of its measure or addressing key parameters that determine the shape of the jump distribution. In contrast, this paper, through a L\'evy measure with a clear analytical form, identifies and systematically analyzes two core structural parameters, \( \beta \) and \( \lambda \). This allows us to quantitatively address a question that past models have struggled with: how does the statistical pattern of human intervention activities shape optimal harvesting strategies and determine the system's sustainable output potential?

Based on this motivation, this paper establishes the following stochastic logistic
population model driven by Brownian motion and L\'evy jump noise with a one-sided
tempered stable L\'evy measure:
\begin{equation}\label{f01}
\mathrm{d}x(t)
=
x(t)[a-h-bx(t)]dt
+\tau x(t)\mathrm{d}B(t)
+\sigma x(t-)\mathrm{d}L(t),
\end{equation}
where \(x(t)\) denotes the population size, \(a>0\) is the intrinsic growth rate,
\(h>0\) is the per capita harvesting rate, \(b>0\) is the intra-species competition
coefficient, and \(\tau^2\) is the intensity of the continuous environmental noise.
Here \(B(t)\) is a standard Brownian motion, \(\sigma>0\) measures the intensity
of the jump noise, and \(x(t-)\) denotes the left-hand limit of \(x(t)\) at time \(t\).
The jump process \(L(t)\) is assumed to be a spectrally positive L\'evy process with
canonical small-jump compensation, namely,
\[
\mathrm{d}L(t)
=
\int_0^{+\infty} z\,N(\mathrm{d}t,\mathrm{d}z)
-
\int_0^1 z\,\nu(\mathrm{d}z)\,\mathrm{d}t,
\]
where \(N(\mathrm{d}t,\mathrm{d}z)\) is a Poisson random measure on \((0,\infty)\) with intensity
measure \(\nu(\mathrm{d}z)\mathrm{d}t\), and \(\nu\) is the L\'evy measure given by (\ref{eq:levy-measure}). Equivalently,
\[
\mathrm{d}L(t)
=
\int_0^{+\infty} z\,\widetilde N(\mathrm{d}t,\mathrm{d}z)
+
\int_1^{+\infty} z\,\nu(\mathrm{d}z)\,\mathrm{d}t,
\]
where \(\widetilde N(\mathrm{d}t,\mathrm{d}z)=N(\mathrm{d}t,\mathrm{d}z)-\nu(\mathrm{d}z)\mathrm{d}t\) is the compensated Poisson
random measure. Throughout this paper, we assume that \(B(t)\) and \(N(\mathrm{d}t,\mathrm{d}z)\),
or equivalently \(B(t)\) and \(\widetilde N(\mathrm{d}t,\mathrm{d}z)\), are independent.

Within the above framework, the central scientific question addressed in this paper is to determine the optimal harvesting effort (OHE) $h^*$ that maximizes the expected sustainable yield (ESY)
\begin{equation*}
Y(h) = \lim_{t \to +\infty} \mathbb{E}\bigl[h\,x(t)\bigr],
\end{equation*}
subject to the constraint of population persistence. To solve this problem, this paper explores the following interrelated aspects:

\begin{enumerate}
    \item The impact mechanism of the control variable: how does the harvesting effort \( h \), as a direct management decision variable, affect the long-term dynamics of the population size \( x(t) \) and ESY \( Y(h) \)? This forms the theoretical foundation for solving OHE \( h^* \).

    \item The regulatory effect of traditional noise: how do the intensities of continuous environmental noise (\( \tau^2 \)) and jump noise (\( \sigma \)) affect OHE \( h^* \) and the maximum expected sustainable yield (MESY) \( Y^* \)?

    \item The key role of the intervention statistical structure: in particular, how do the statistical structure parameters of human intervention-\( \beta \) and \( \lambda \)-quantitatively alter \( h^* \) and \( Y^* \)? That is, how do optimal strategies respond when the management paradigm shifts from ``high-frequency small-scale'' (small \( \beta \)) to "low-frequency large-scale" (large \( \beta \)), or from ``allowing large-scale emergency responses'' (small \( \lambda \)) to ``scale-limited'' (large \( \lambda \))?
\end{enumerate}

Through this series of analyses, this paper aims to uncover a core mechanism: in a complex stochastic environment, the intrinsic statistical structure of management activities (characterized by \( \beta, \lambda \)) and traditional environmental fluctuations (characterized by \( \tau, \sigma \)) jointly form a ``random force field'', which not only determines the harvestability of the population but also fundamentally requires the optimal management strategy \( h^* \) to match it. Understanding the impact of \( h \) on \( x(t) \) and \( Y(h) \) is the key to understanding how this ``matching'' can be achieved.

This paper is structured as follows: Section 2 proves the existence of a unique positive solution to the mixed noise-driven logistic model (Lemma 2.1) and establishes the sufficient conditions for extinction and persistence (Theorem 2.1); Section 3 proves that the model exhibits distributional stability (Theorem 3.1); Section 4 derives explicit expressions for OHE and MESY (Theorem 4.1); Section 5 conducts numerical simulations to validate the theoretical results; Section 6 provides a comprehensive analysis and summary of the paper and discusses several open problems for future research. This study aims to provide deeper theoretical insights and decision-making support for the robust, adaptive, and sustainable use of renewable resources in complex, multi-scale environments influenced by institutional human interventions.
\medskip

\rm
\section{Extinction and Persistence}
\label{2}

\medskip

In this paper, we consider a complete probability space \((\Omega, \mathcal{F}, \{\mathcal{F}_t\}_{t \geq 0}, P)\), where the filtration \(\{\mathcal{F}_t\}_{t \geq 0}\) is the filtration generated by Brownian motion $B(t)$ and the Poisson random measure  $N({\rm d}t,{\rm d}z)$,
and augmented to satisfy the usual conditions (i.e., it is right-continuous and complete). We denote the set of positive real numbers by \(\mathbb{R}_+\). As motivated in the Introduction, the positive jumps in the model represent stochastic interventions (e.g., restocking) with random timing and magnitude. Under the It\^{o}-L\'evy framework, the compensated jump measure introduces a drift correction term, so the net effect of jumps on population growth is not necessarily positive.

This section is devoted to the theoretical analysis of the dynamics of model (\ref{f01}). We first prove the existence and uniqueness of a global positive solution (Lemma 2.1), and then, by introducing a key threshold parameter \(\Phi\), we establish sufficient conditions that determine the extinction or persistence of the population (Theorem 2.1).

\medskip

\begin{lemma}
For any initial data $x(0) \in \mathbb{R}_+ $, model (\ref{f01}) has a unique global positive solution $x(t)$ for all $t\geq 0$ almost surely (\text{a.s.}). That is to say,
\begin{equation*}
P\{x(t) \in \mathbb{R}_+\;\text{\rm for all }t \geq 0\} = 1.
\end{equation*}
\end{lemma}
\medskip
\begin{proof}[$\mathbf{Proof\ of\ Lemma\ 2.1}$] \rm Note that  the coefficients of the model (\ref{f01}) are locally Lipschitz continuous on $ \mathbb{R}_+ $, therefore, for any given initial value $ x(0) \in \mathbb{R}_+ $, there exists a unique local solution $ x(t) $, defined on the time interval $[0, \tau_e)$, where  $ \tau_e = \inf{\{t>0: \lim\sup_{s\to t-}|x(s)|=\infty\}} $ is the explosion time of \((x(t))\).

Next, we prove $ \tau_e = \infty $ a.s. Let \( k_0 > 0 \) such that \( x(0) < k_0 \). For each integer \( k > k_0 \), define:

\begin{equation*}
\begin{array}{rcl}
\tau_k := \inf \left\{ t \in [0, \tau_e) : x(t) \notin \left( \frac{1}{k}, k \right) \right\},
\end{array}
\end{equation*}
with the convention \(\inf \{\emptyset\} = \infty\). Clearly, \(\tau_k\) is increases with \(k\), and Set \(\tau_\infty = \lim_{k \to \infty} \tau_k\), so \(\tau_\infty \leq \tau_e\) a.s. If we show \(\tau_\infty = \infty\) a.s., then \(\tau_e = \infty\) a.s.

\noindent Consider the Lyapunov function:
\begin{equation*}
\begin{array}{rcl}
V_1(x) = x^\theta - \ln x - \theta^{-1}(1 + \ln \theta), \quad x \in \mathbb{R}_+,\quad \theta \in (0, 1),
\end{array}
\end{equation*}
which is nonnegative. For any \( T > 0 \) and \( t \in [0, T] \), applying the It\^{o}-L\'evy formula:

\begin{equation}\label{f18}
\begin{array}{rcl}
\ds V_1(x(t \wedge \tau_k)) = V_1(x(0)) + \int_0^{t \wedge \tau_k} \mathcal{L}V_1(x(s)) \, \mathrm{d}s + M_1(t \wedge \tau_k),
\end{array}
\end{equation}
where
\begin{equation*}
\begin{array}{rcl}
\ds M_1(t \wedge \tau_k)&=&\ds\int_0^{t \wedge \tau_k} \left(\theta\tau x^{\theta}(s)-\tau )  \right)\, \mathrm{d}B(s)\\
&\quad&\ds+\int_0^{t \wedge \tau_k}\int_0^{+\infty} \left[ x^\theta(s^{-})((1+\sigma z)^\theta - 1) - \ln(1+\sigma z) \right] \widetilde{N}(\mathrm{d}s,\,\mathrm{d}z).\\
\end{array}
\end{equation*}
The generator \( \mathcal{L}_1V(x) \) is:
\begin{equation*}
\begin{array}{rcl}
\mathcal{L}V_1(x) &=&\ds (\theta x^{\theta} - 1) \cdot (a - h+\sigma
\int_1^{+\infty} z\,\nu(\mathrm{d}z)- bx ) + \frac{1}{2} \tau ^2 \cdot (\theta(\theta-1)x^{\theta} + 1) \\
&\quad&\ds+ \int_0^{+\infty} \left[ V_1(x + \sigma xz) - V_1(x) - V_1^{'}(x) \cdot \sigma xz \right] \nu(\mathrm{d}z)\\
&=& (\theta x^{\theta} - 1) \cdot (a - h- bx ) + \frac{1}{2} \tau ^2 \cdot (\theta(\theta-1)x^{\theta} + 1) \\
&\quad&\ds+ \int_0^{+\infty} \left[ V_1(x + \sigma xz) - V_1(x) - V_1^{'}(x) \cdot \sigma xz\mathbf{1}_{\{0<z\leq 1\}}  \right] \frac{e^{-\lambda z}}{z^{1+\beta}} \, \mathrm{d}z,\\
\end{array}
\end{equation*}
After simplification:
\begin{equation*}
\begin{array}{rcl}
\mathcal{L}V_1(x) &=& a\theta x^\theta - h\theta x^\theta- b\theta x^{\theta+1}   - a + h + bx+ \frac{1}{2}\tau^2\theta(\theta-1)x^\theta + \frac{1}{2}\tau^2 \\
&\quad&\ds+ \int_0^{+\infty} \left[ x^\theta((1+\sigma z)^\theta - 1) - \ln(1+\sigma z) - (\theta x^{\theta} - 1)\sigma z\mathbf{1}_{\{0<z\leq 1\}}  \right] \frac{ e^{-\lambda z}}{z^{1+\beta}} \, \mathrm{d}z, \\
\end{array}
\end{equation*}
Let $y := \sigma z$. Then,
\begin{equation*}
\begin{array}{rcl}
\mathcal{L}V_1(x)
&=& \ds a\theta x^\theta - h\theta x^\theta- b\theta x^{\theta+1}   - a + h + bx+ \frac{1}{2}\tau^2\theta(\theta-1)x^\theta + \frac{1}{2}\tau^2 \\
&\quad&\ds+ C_1+C_2+(C_3+C_4)x^{\theta}, \\
\end{array}
\end{equation*}
where
\begin{equation*}
\begin{array}{rcl}
\ds C_1 := \sigma^\beta \int_0^\sigma (y - \ln(1 + y)) \frac{e^{-\lambda y/\sigma}}{y^{1+\beta}} \,\mathrm{d}y,
\end{array}
\end{equation*}
\begin{equation*}
\begin{array}{rcl}
\ds C_2 := \sigma^\beta \int_\sigma^{+\infty} \ln(1 + y) \frac{ e^{-\lambda y/\sigma}}{y^{1+\beta}}\, \mathrm{d}y,
\end{array}
\end{equation*}
\begin{equation*}
\begin{array}{rcl}
\ds C_3 := \sigma^\beta \int_0^\sigma \left( (1 + y)^\theta - 1 - \theta y \right) \frac{e^{-\lambda y/\sigma}}{y^{1+\beta}} \,\mathrm{d}y,
\end{array}
\end{equation*}
\begin{equation*}
\begin{array}{rcl}
\ds C_4 := \sigma^\beta \int_\sigma^{+\infty} \left[ (1 + y)^\theta - 1 \right] \frac{e^{-\lambda y/\sigma}}{y^{1+\beta}} \,\mathrm{d}y.
\end{array}
\end{equation*}
Define \( H_1(y) := \frac{y - \ln(1 + y)}{y^2}\) for \( y \in (0, \sigma] \). Elementary calculus shows
\begin{equation*}
\begin{array}{rcl}
\ds\frac{\sigma - \ln(1 + \sigma)}{\sigma^2} \leq H_1(y) \leq \frac{1}{2}, \quad y \in (0, \sigma].
\end{array}
\end{equation*}
So \( H_1(y) \) is bounded by a constant \( \bar{C}_1>0\). Since
$
0 < e^{-\lambda y/\sigma} \le 1, \; y \in (0,\sigma],
$
then
\begin{equation*}
C_1 \le \sigma^\beta \cdot \bar{C}_1 \cdot \int_0^\sigma \frac{1}{y^{\beta - 1}}  \,\mathrm{d}y
=\frac{\, \bar{C}_1 \, \sigma^2}{2 - \beta} < \infty.
\end{equation*}
For $C_2$, define
\[
H_2(y):=e^{-\lambda y/\sigma}\frac{\ln(1+y)}{y^{1+\beta}},\qquad
\bar{H}_2(y):=\frac{1}{y^{q}},\quad q>1,\ y\ge \sigma.
\]
Clearly, $H_2(y)\ge0$ and $\bar{H}_2(y)>0$. Moreover,
\[
\lim_{y\to\infty}\frac{H_2(y)}{\bar{H}_2(y)}
=\lim_{y\to\infty}y^{q}e^{-\lambda y/\sigma}\frac{\ln(1+y)}{y^{1+\beta}}=0.
\]
Since $\int_{\sigma}^{\infty} \bar{H}_2(y)\,\mathrm{d}y<\infty$ for $q>1$, the limit comparison test implies that $\int_{\sigma}^{\infty} H_2(y)\,\mathrm{d}y<\infty$. Hence $C_2<\infty$.
By bounding \( C_3 \) (similar to \( C_1 \)) and noting \( C_4 \) converges (similar to \( C_2 \)), we find a positive constant \( C>0\) such that
\begin{equation*}
\begin{array}{rcl}
\ds \mathcal{L}V_1(x) \leq C \quad \text{for all } x \in \mathbb{R}_+.
\end{array}
\end{equation*}
Thus, from (\ref{f18}):
\begin{equation}\label{f28}
\begin{array}{rcl}
\ds \mathbb{E}V_1(x(T \wedge \tau_k)) \leq V_1(x(0)) + C\mathbb{E}(\tau_k \wedge T) \leq V_1(x(0)) + CT.
\end{array}
\end{equation}
Now suppose for contradiction that \( P(\tau_\infty \leq T) > 0 \). Then for sufficiently large \( k \), \( P(\tau_k \leq T)> 0 \). On the event \(\omega \in \{\tau_k \leq T\}\), we have
\begin{equation*}
\begin{array}{rcl}
\ds x(\tau_k, \omega) \geq k \quad \text{or} \quad x(\tau_k, \omega) \leq \frac{1}{k}.
\end{array}
\end{equation*}
Since \( V_1(x) \) decreases on \( x \in (0, (1/\theta)^{1/\theta}) \) and increases on \( x \in [(1/\theta)^{1/\theta}, +\infty) \),
\begin{equation}\label{f30}
\begin{array}{rcl}
\ds V_1(x(\tau_k, \omega)) \geq \left( k^\theta - \log k - \theta^{-1}(1 + \log \theta) \right) \wedge \left( \frac{1}{k^\theta} + \log k - \theta^{-1}(1 + \log \theta) \right).
\end{array}
\end{equation}
Combining (\ref{f28}) and (\ref{f30}):
\begin{equation*}
\begin{array}{rcl}
&\left( k^\theta - \log k - \theta^{-1}(1 + \log \theta) \right) \wedge \left( \frac{1}{k^\theta} + \log k - \theta^{-1}(1 + \log \theta) \right) P(\tau_k \leq T)\\
&\leq \ds \mathbb{E}(V_1(x(\tau_k, \omega)) I_{\{\tau_k \leq T\}}) \leq V_1(x(0)) + CT.\\
\end{array}
\end{equation*}
Letting \( k \to \infty \), the left side tends to infinity unless \(P(\tau_k \leq T) \to 0\), which is a contradiction. Hence,
\begin{equation*}
\begin{array}{rcl}
P(\tau_\infty = \infty) = 1.
\end{array}
\end{equation*}
By the definition of $\tau_k$ and $\tau_\infty = \infty$ a.s., we have $x(t) > 0$ for all $t \ge 0$ a.s.

Therefore, model (\ref{f01}) admits a unique global positive solution \( x(t) \in \mathbb{R}_+ \) on \( t \geq 0 \).
\end{proof}
\medskip

\rm
\begin{lemma}
(Liu, Wang and Wu \citep{a3}).
Let $\Psi(t)$ be a continuous process valued on $\mathbb{R}_+$.
\medskip

\noindent$\rm(\mathfrak{A})$ If there exist some constants $T>0$, $\lambda_0>0$, $\lambda$, $\sigma_i$ and $\lambda_{i}$ such that for all $t\geq T$,
\begin{equation*}
\ln \Psi(t)\leq \lambda t-\lambda_0\int_0^t\Psi(s)\,\mathrm{d}s+\sum_{i=1}^n\sigma_i B_i(t)+\sum_{i=1}^n\lambda_i\int_0^t\int_0^{+\infty}\ln(1+\gamma_{i}(v))\widetilde{N}(\mathrm{d}s,\,\mathrm{d}v),\ a.s.,
\end{equation*}
then
\begin{equation*}
\left\{\begin{array}{l}
\ds\limsup_{t\rightarrow+\infty}t^{-1}\int_0^t\Psi(s)\,\mathrm{d}s\leq \lambda/\lambda_0,\quad a.s., \quad if\ \lambda\geq0,\\
\ds\lim_{t\rightarrow+\infty}\Psi(t)=0,\quad a.s., \quad if\ \lambda<0.
\end{array}
\right.
\end{equation*}
$\rm(\mathfrak{B})$ If there exist some constants $T>0$, $\lambda_0>0$, $\lambda>0$, $\sigma_i$ and $\lambda_{i}$ such that for all $t\geq T$,
\begin{equation*}
\ln \Psi(t)\geq\lambda t-\lambda_0\int_0^t\Psi(s)\mathrm{d}s+\sum_{i=1}^n\sigma_i B_i(t)+\sum_{i=1}^n\lambda_i\int_0^t\int_0^{+\infty}\ln(1+\gamma_{i}(v))\widetilde{N}(\mathrm{d}s,\,\mathrm{d}v),\ a.s.,
\end{equation*}
then
\begin{equation*}
\liminf_{t\rightarrow+\infty}t^{-1}\int_0^t\Psi(s)\,\mathrm{d}s\geq \lambda/\lambda_0,\quad a.s.\\
\quad
\end{equation*}
\end{lemma}
\medskip

\begin{theorem}
For model (\ref{f01}), $\beta \in (0,2)$, define:
\[
\Phi := a - h - \frac{\tau^2}{2} + \int_0^{+\infty} \left[\ln(1 + \sigma z) - \sigma z \mathbf{1}_{\{0<z\leq 1\}} \right] \frac{e^{-\lambda z}}{z^{1+\beta}} \, \mathrm{d}z.
\]

Then the following conclusions hold:
\begin{enumerate}[
    label=(\roman*),
    leftmargin=*,
    align=left,
    itemindent=0pt,
    labelsep=0pt
]
    \item If $\Phi < 0$, then $\lim_{t \to \infty} x(t) = 0$ a.s., i.e., the population goes extinct.
    \item If $\Phi = 0$, then $\lim_{t \to \infty} \frac{1}{t} \int_0^t x(s)\, \mathrm{d}s = 0$ a.s., i.e., the population is non-persistent in time average a.s.
    \item If $\Phi > 0$, then $\lim_{t \to \infty} \frac{1}{t} \int_0^t x(s)\, \mathrm{d}s = \frac{\Phi}{b}$ a.s., i.e., the population is stable in time average a.s.
\end{enumerate}

\end{theorem}
\medskip

\begin{proof}[$\mathbf{Proof\ of\ Theorem\ 2.1}$] \rm
Apply It\^{o}-L\'evy formula to $\ln x(t)$:
\begin{align*}
\mathrm{d}\ln x(t) 
&= \left(a - h +\sigma\int_1^{+\infty} z\,\nu(\mathrm{d}z)- b x(t) - \frac{\tau^2}{2}\right)\,\mathrm{d}t + \tau\,\mathrm{d}B(t) \\
&\quad + \int_0^{+\infty} \left[\ln(1 + \sigma z) - \sigma z \right] \nu(\mathrm{d}z) \mathrm{d}t \\
&\quad + \int_0^{+\infty} \ln(1 + \sigma z)\, \widetilde{N}(\mathrm{d}t,\, \mathrm{d}z) \\
&= \left(a - h - b x(t) - \frac{\tau^2}{2}\right)\,\mathrm{d}t + \tau\,\mathrm{d}B(t) \\
&\quad + \int_0^{+\infty} \left[\ln(1 + \sigma z) - \sigma z \mathbf{1}_{\{0<z\leq 1\}} \right] \frac{ e^{-\lambda z}}{z^{1+\beta}} \, \mathrm{d}z \mathrm{d}t \\
&\quad + \int_0^{+\infty} \ln(1 + \sigma z)\, \widetilde{N}(\mathrm{d}t,\, \mathrm{d}z) \\
&= (\Phi - b x(t))\,\mathrm{d}t + \tau\,\mathrm{d}B(t) + \int_0^{+\infty} \ln(1 + \sigma z)\, \widetilde{N}(\mathrm{d}t,\,\mathrm{d}z),
\end{align*}
Integrating yields:
\begin{equation}\label{f34}
\begin{array}{rcl}
\ds\ln x(t) = \ln x(0) + \Phi t - b \int_0^t x(s)\, \mathrm{d}s + \tau B(t) + \int_0^t \int_0^{+\infty} \ln(1 + \sigma z)\, \widetilde{N}(\mathrm{d}s,\,\mathrm{d}z).
\end{array}
\end{equation}
Divide both sides by $t$:
\begin{equation}\label{f35}
\begin{array}{rcl}
\ds\frac{\ln x(t)}{t} = \frac{\ln x(0)}{t} + \Phi - \frac{b}{t} \int_0^t x(s)\, \mathrm{d}s + \frac{\tau B(t)}{t} + \frac{M_2(t)}{t},
\end{array}
\end{equation}
where
\[
M_2(t) = \int_0^t \int_0^{+\infty} \ln(1 + \sigma z)\, \widetilde{N}(\mathrm{d}s,\,\mathrm{d}z).
\]
The Meyer's angle bracket process of $M_2(t)$ is
\[
\langle M_2, M_2 \rangle(t) = \int_0^t \int_0^{+\infty} (\ln(1 + \sigma z))^2 \frac{e^{-\lambda z}}{z^{1+\beta}} \, \mathrm{d}z \mathrm{d}s,
\]
for $\beta \in (0,2)$. By splitting the integral at $z=1$ and arguing similarly to the proofs of $C_1$ and $C_2$,
noting that $(\ln(1+\sigma z))^2$ behaves like $z^2$ as $z\to0$ and grows at most polynomially
as $z\to\infty$, we conclude that
\begin{equation*}
\begin{array}{rcl}
\ds\int_0^{+\infty} (\ln(1 + \sigma z))^2 \frac{e^{-\lambda z}}{z^{1+\beta}} \, \mathrm{d}z < +\infty.
\end{array}
\end{equation*}
We have
\begin{equation}\label{f36}
\begin{array}{rcl}
\ds\int_0^t \frac{\mathrm{d}\langle M_2, M_2\rangle(s)}{(1+s)^2} = \frac{t}{1+t} \int_0^{+\infty} (\ln(1 + \sigma z))^2 \frac{e^{-\lambda z}}{z^{1+\beta}} \,\mathrm{d}z< +\infty.
\end{array}
\end{equation}
Combining (\ref{f36}) with the strong law of large numbers for local martingales \citep{a31}, we therefore have
\begin{equation*}
\begin{array}{rcl}
\ds\lim_{t \to \infty} \frac{M_2(t)}{t} = 0 \quad \text{a.s.}
\end{array}
\end{equation*}
Additionally, it is well known that
\begin{equation*}
\begin{array}{rcl}
\ds\lim_{t \to \infty} \frac{B(t)}{t} = 0
\end{array}
\end{equation*}

\noindent $(\romannumeral 1)$ Extinction $(\Phi < 0)$: From (\ref{f35}),
\[
\ds\frac{\ln x(t)}{t} \leq \frac{\ln x(0)}{t} + \Phi + \frac{\tau B(t)}{t} + \frac{M_2(t)}{t}.
\]
Letting $t \to \infty$,
\[
\limsup_{t \to \infty} \frac{\ln x(t)}{t} \leq \Phi.
\]
Due to $\Phi < 0$, we obtain
\[
\lim_{t \to \infty} x(t) = 0, \quad \text{a.s.}
\]

\noindent $(\romannumeral 2)$ Non-persistence $(\Phi = 0)$ and $(\romannumeral 3)$ Persistence ($\Phi > 0$): Since
\[
\lim_{t \to \infty} \frac{\ln x(0)}{t} = 0,
\]
thus for arbitrary $\varepsilon > 0$, there exists a random time $T_1 > 0$ such that for $t \geq T_1$,
\[
-\varepsilon \leq \frac{\ln x(0)}{t} \leq \varepsilon.
\]
From \eqref{f34},
\begin{equation}\label{f39}
\begin{array}{rcl}
\ds\ln x(t) &\leq \ds(\Phi + \varepsilon)t - b \int_0^t x(s)\, \mathrm{d}s + \tau B(t) + \int_0^t \int_0^{+\infty} \ln(1 + \sigma z)\, \widetilde{N}(\mathrm{d}s, \,\mathrm{d}z), \end{array}
\end{equation}
\begin{equation}\label{f40}
\begin{array}{rcl}
\ds\ln x(t) &\geq \ds(\Phi - \varepsilon)t - b \int_0^t x(s)\, \mathrm{d}s + \tau B(t) + \int_0^t \int_0^{+\infty} \ln(1 + \sigma z)\, \widetilde{N}(\mathrm{d}s, \,\mathrm{d}z). 
\end{array}
\end{equation}
If $\Phi = 0$, Applying $\rm(\mathfrak{A})$ in Lemma 2.2 to \eqref{f39}, we can deduce that
\begin{equation*}
\ds\lim_{t \to \infty} \frac{1}{t} \int_0^t x(s)\, \mathrm{d}s\leq\limsup_{t\rightarrow+\infty}\frac{1}{t}\int_0^tx(s)\,\mathrm{d}s\leq \frac{\varepsilon}{b},\quad \text{a.s.}.
\end{equation*}
From the arbitrariness of $\varepsilon$, then (ii) holds.

\noindent If $\Phi > 0$, let $\varepsilon$ be sufficiently small such that $\Phi - \varepsilon>0$. Applying $\rm(\mathfrak{A})$ and $\rm(\mathfrak{B})$ in Lemma 2.2 to \eqref{f39} and \eqref{f40}, we observe that
\[
\ds\frac{(\Phi - \varepsilon)}{b} \leq \liminf_{t \to \infty} \frac{1}{t} \int_0^t x(s)\, \mathrm{d}s \leq \limsup_{t \to \infty} \frac{1}{t} \int_0^t x(s)\, \mathrm{d}s \leq \frac{(\Phi + \varepsilon)}{b}.
\]
From the arbitrariness of $\varepsilon$, then (iii) holds.
\end{proof}

\rm
\section{Stationary in Distribution}
\label{3}

In this section, we will give some sufficient conditions which guarantee that the model (\ref{f01}) has a unique stationary distribution.
\medskip

\begin{lemma}
If \( \beta \in (0,2) \), Then for any $p > 0$, there exists a positive $K_p$ such that $x(t)$ of the model (\ref{f01}) satisfies
\begin{equation*}
\limsup_{t\to\infty} \mathbb{E}\left[(x(t))^{p}\right] \leq K_p.
\end{equation*}
\end{lemma}
\medskip

\begin{proof}[$\mathbf{Proof\ of\ Lemma\ 3.1}$]
Define $V_2(x) := x^{p}$ for $p > 0$. Applying the It\^{o}-L\'evy formula to $V(x)$ yields
\[
\mathrm{d}V_2(x(t)) = \mathcal{L}V_2(x(t))\,\mathrm{d}t + p x^{p}(t)\tau \,\mathrm{d}B(t) + \int_{0}^{+\infty} \left[(x(t^-) + \sigma x(t^-)z)^{p} - (x(t^-))^{p}\right] \widetilde{N}(\mathrm{d}t,\,\mathrm{d}z),
\]
where the generator $\mathcal{L}V(x)$ is given by

\begin{equation*}
\begin{array}{rcl}
\mathcal{L}V_2(x)& =&\ds p x^{p}\left[a-h+\sigma\int_1^{+\infty} z\,\nu(\mathrm{d}z)-bx + \frac{p-1}{2}\tau^{2}\right] \\
&\quad& \ds+ x^{p} \int_{0}^{+\infty} \left[(1+\sigma z)^{p} - 1 - p\sigma z \right]\,\nu(\mathrm{d}z),\\
& =&\ds p x^{p}\left[a-h-bx + \frac{p-1}{2}\tau^{2}\right] + x^{p} J(\sigma,p),\\
\end{array}
\end{equation*}
with
\[
J(\sigma,p) = \int_{0}^{+\infty} \left[(1+\sigma z)^{p} - 1 - p\sigma z\mathbf{1}_{\{0<z\leq 1\}} \right]\,\nu(\mathrm{d}z).
\]
We now estimate $J(\sigma,p)$. Split the integral into two parts:
\[
J(\sigma,p) =  \underbrace{\int_{0}^{1} \left[(1+\sigma z)^{p} - 1 - p\sigma z\right]\, \nu(\mathrm{d}z)}_{\ds J_1} + \underbrace{\int_{1}^{+\infty} \left[(1+\sigma z)^{p} - 1\right]\, \nu(\mathrm{d}z)}_{\ds J_2}.
\]
\noindent For $J_1$: Let $F(z) := (1+\sigma z)^{p} - 1 - p\sigma z$. By Taylor's theorem, there exists $\theta \in (0,1)$ such that
\[
|F(z)| = \frac{|p(p-1)|}{2}(1+\theta\sigma z)^{p-2}(\sigma z)^{2},\quad\sigma>0,\quad 0<z<1 .
\]
Note that for $z\in(0,1]$ we have
\[
1 \le 1+\theta\sigma z \le 1+\sigma .
\]
If $0<p<2$, then $p-2<0$ and the function $(1+\theta\sigma z)^{p-2}$ is monotone decreasing. Since $1+\theta\sigma z \ge 1$, it follows that
\[
(1+\theta\sigma z)^{p-2} \le 1 .
\]
If $p\ge 2$, then $p-2\ge 0$ and the function $(1+\theta\sigma z)^{p-2}$ is monotone increasing. Hence,
\[
(1+\theta\sigma z)^{p-2} \le (1+\sigma)^{p-2}.
\]
Define
\[
K(\sigma,p):=\max\{1,(1+\sigma)^{p-2}\},
\]
then for all $p>0$ and all $z\in(0,1]$,
\[
|F(z)|
\le \frac{|p(p-1)|}{2}\,K(\sigma,p)\,\sigma^2 z^2.
\]
Therefore,
\[
J_1\le\int_0^1 |F(z)|\,\nu(\mathrm{d}z)
\le \frac{|p(p-1)|}{2}\,K(\sigma,p)\,\sigma^2
\int_0^1 z^2\,\nu(\mathrm{d}z).
\]
Substituting the expression of $\nu(dz)$ and $\beta \in (0,2)$, we obtain
\[
\int_0^1 z^2\,\nu(dz)
= \int_0^1 z^2 \frac{e^{-\lambda z}}{z^{1+\beta}}\,\mathrm{d}z
= \int_0^1 z^{1-\beta} e^{-\lambda z}\,\mathrm{d}z
\le \int_0^1 z^{1-\beta}\,\mathrm{d}z
= \frac{1}{2-\beta}.
\]
Combining the above estimates yields
\begin{equation}\label{bc1}
J_1\le \frac{|p(p-1)|}{2}\,K(\sigma,p)\,\sigma^2\cdot\frac{1}{2-\beta}:=K^{'}.
\end{equation}
\noindent For $J_2$: If $0<p\le 1$, it is obvious that
\[
(1+\sigma z)^p-1 \le \sigma^p z^p, \quad\sigma>0,\quad z>1  .
\]
If $p\ge 1$, we use the classical inequality
\((a+b)^p \le 2^{p-1}(a^p+b^p)\),
\(a,b\ge 0.\)
Taking $a=1$ and $b=\sigma z$, it follows that
\[
(1+\sigma z)^p \le 2^{p-1}(1+\sigma^p z^p),
\]
and thus
\[
(1+\sigma z)^p-1 \le 2^{p-1}\sigma^p z^p + \bigl(2^{p-1}-1\bigr).
\]
Combining the above two cases, there exist constants depending only on $p$,
\[
C_p := \max\{1,\,2^{p-1}\},
\qquad
D_p := \max\{0,\,2^{p-1}-1\},
\]
such that,
\[
(1+\sigma z)^p - 1 \le C_p \sigma^p z^p + D_p .
\]
Therefore,
\begin{equation*}
\begin{array}{rcl}
J_2&=&\ds \int_1^\infty \bigl[(1+\sigma z)^p - 1\bigr]\,\nu(dz)\\
&\le&\ds \int_1^\infty \bigl(C_p \sigma^p z^p + D_p\bigr)\frac{e^{-\lambda z}}{z^{1+\beta}}\,dz \\
&\le& \ds C_p \sigma^p
\int_1^\infty z^{p-1-\beta} e^{-\lambda z}\,dz
+
D_p
\int_1^\infty z^{-1-\beta} e^{-\lambda z}\,dz .
\end{array}
\end{equation*}
Since $e^{-\lambda z}$ exhibits exponential decay, both integrals on the right-hand side converge for any real exponent, and thus the right-hand side is finite. Therefore, $J_2$ is bounded. That is, there exists a constant $K^{''}>0$ such that
\begin{equation}\label{bc2}
J_2
\le K^{''}.
\end{equation}
From (\ref{bc1}) and (\ref{bc2}), we have
\[
J(\sigma,p)=J_1+J_2\leq K^{'}+K^{''}:=\widehat{K}.
\]
Substituting this bound, we get
\[
\mathcal{L}V_2(x) \leq p x^{p}\left[a-h-bx + \frac{p-1}{2}\tau^{2}\right] + x^{p} \widehat{K}.
\]
Hence,
\begin{equation*}
\begin{array}{rcl}
V_2(x) + \mathcal{L}V_2(x) &\leq&\ds x^{p}\left[1 + \widehat{K} + p\left(a-h + \frac{p-1}{2}\tau^{2}\right) - p b x\right]\\
&\leq&\ds x^{p}\left[1 + \widehat{K} + p\left(a-h + \frac{p-1}{2}\tau^{2}\right)\right]- p b x^{p+1}.
\end{array}
\end{equation*}
Then there exists a finite maximum $K_p > 0$ such that
\[
V_2(x) + \mathcal{L}V_2(x) \leq K_p, \quad \forall x \in \mathbb{R}_+.
\]
Define the stopping time $\sigma_k := \inf\{t \geq 0 : x(t) > k\}$. By the global existence of the solution, $\sigma_k \to \infty$ almost surely as $k \to \infty$. Applying It$\rm\hat{o}$' formula to $Y(t) = e^{t} V_2(x(t))$ and taking expectation gives
\[
\mathbb{E}\left[e^{t \wedge \sigma_k} V_2(x(t \wedge \sigma_k))\right] = V_2(x(0)) + \mathbb{E}\left[\int_{0}^{t \wedge \sigma_k} e^{s} \left[V_2(x(s)) + \mathcal{L}V_2(x(s))\right]\, \mathrm{d}s\right]
\]
\[
\leq V_2(x(0)) + K_p \mathbb{E}\left[\int_{0}^{t \wedge \sigma_k} e^{s}\, \mathrm{d}s\right] = V_2(x(0)) + K_p (e^{t \wedge \sigma_k} - 1).
\]
Letting $k \to \infty$ and applying Fatou's lemma, we obtain
\[
\mathbb{E}\left[e^{t} V_2(x(t))\right] \leq V_2(x(0)) + K_p (e^{t} - 1),
\]
which implies
\[
\mathbb{E}[x^{p}(t)] \leq e^{-t} V_2(x(0)) + K_p (1 - e^{-t}).
\]
Therefore,
\[
\limsup_{t \to \infty} \mathbb{E}[x^{p}(t)] \leq K_p.
\]
\end{proof}
\medskip

\begin{theorem}
Model (\ref{f01}) is  stable in distribution, i.e., there is a unique probability measure $\pi(\cdot)$ on $[0,+\infty)$ such that, for every initial data $x(0)\in \mathbb{R}_+$, the transition probability $p(t,x(0),\cdot)$ of $x(t)$ converges weakly to $\pi(\cdot)$ as $t\rightarrow+\infty$.
\end{theorem}

\begin{proof}[$\mathbf{Proof\ of\ Theorem\ 3.1}$] \rm Let $x(t)$ and $\widetilde{x}(t)$ be two solutions of model (\ref{f01}) with initial data $x(0)\in \mathbb{R}_+$ and $\widetilde{x}(0)\in \mathbb{R}_+$, respectively. Define
\begin{equation*}
\begin{array}{rcl}
\ds V(t)&:=&\ds\big|\ln x(t)-\ln \widetilde{x}(t)\big|.\\
\end{array}
\end{equation*}
Applying Itô-Lévy formula, we obtain
\[
\mathrm{d}^+V(t)=\sgn\big(x(t)-\widetilde{x}(t)\big)\big[\mathrm{d}\ln x(t)-\mathrm{d}\ln \widetilde{x}(t)]\big|,
\]
where $\mathrm{d}^+$ means right derivative.
Since the diffusion terms and jump terms cancel completely, we have
\[
\mathrm{d}\ln x(t)-\mathrm{d}\ln \widetilde{x}(t)
=
-b(x(t)-\widetilde{x}(t))\,\mathrm{d}t.
\]
Hence
\[
\mathrm{d}^+V(t)
=\ds-b|x(t)-\widetilde{x}(t)|\,\mathrm{d}t.
\]
Consequently,
\begin{equation*}
\ds0\leqslant\mathbb{E}(V(t))=\ds V(0)-b\int_{0}^{t}\mathbb{E}\big|x(s)-\widetilde{x}(s)\big|\,\mathrm{d}s
\end{equation*}
which yields
\begin{equation*}
\ds\sup_{t \geq 0}\int_{0}^{t}\mathbb{E}\big|x(s)-\widetilde{x}(s)\big|\,\mathrm{d}s\leqslant \frac{V(0)}{b}<+\infty,\\
\end{equation*}
Hence
\begin{equation}\label{w}
\mathbb{E}\big|x(t)-\widetilde{x}(t)\big|\in L^1[0,+\infty).
\end{equation}
We now show that $\mathbb{E}|x(t) - \tilde{x}(t)|$ is uniformly continuous.
For $0 \le s < t$, the It$\rm\hat{o}$'-L\'evy dynamics of $x(t)$ gives
\begin{equation*}
\begin{array}{rcl}
x(t) - x(s) &=&\ds \int_s^t x(r)\big[a-h+\sigma \int_{1}^{+\infty} z \, \nu(\mathrm{d}z)-bx(r)\big]\mathrm{d}r
+ \tau\int_s^t x(r)\mathrm{d}B_r\\
&\quad&\ds+ \sigma\int_s^t\int_0^{+\infty} x(r^{-})z\,\widetilde{N}(\mathrm{d}r,\mathrm{d}z).
\end{array}
\end{equation*}
By the inequality \((a+b+c)^2\le 3(a^2+b^2+c^2)\) and It$\rm\hat{o}$ isometries, we obtain
\[
\mathbb{E}|x(t)-x(s)|^2 \le 3\big(\mathbb{E}|I_1|^2+\mathbb{E}|I_2|^2+\mathbb{E}|I_3|^2\big)
\]
with
\begin{equation*}
\begin{array}{rcl}
I_1&=&\ds\int_s^t x(r)[a-h+\sigma \int_{1}^{+\infty} z \, \nu(\mathrm{d}z)-bx(r)]\mathrm{d}r,\\
I_2&=&\ds\tau\int_s^t x(r)\mathrm{d}B_r,\\
I_3&=&\ds\sigma\int_s^t\int_0^{+\infty} x(r^{-})z\,\widetilde{N}(\mathrm{d}r,\mathrm{d}z).
\end{array}
\end{equation*}
Next, we estimate the three terms:

\noindent Drift term \(I_1\): It follows from Lemma~3.1 that for any $p \ge 1$, there exists a constant $K_p > 0$ such that
\[
\sup_{r \ge 0} \mathbb{E}\bigl|x(r)\bigr|^p \le K_p.
\]
Note that
\[
\int_{1}^{+\infty} z\,\nu(dz)
=
\int_{1}^{+\infty}\frac{e^{-\lambda z}}{z^{\beta}}\,dz
<\infty,
\]
and hence
\[
a-h+\sigma\int_{1}^{+\infty} z\,\nu(dz)
\]
is a finite constant.
Therefore, there exists a constant \(\widehat{M}\) such that
\[
\mathbb{E}\big|x(r)[a-h+\sigma \int_{1}^{+\infty} z \, \nu(\mathrm{d}z)-bx(r)]\big|^2 \le \widehat{M},
\]
and consequently
\begin{equation*}
\begin{array}{rcl}
\mathbb{E}|I_1|^2&=&\ds\mathbb{E}\Big|\int_s^t x(r)[a-h+\sigma \int_{1}^{+\infty} z \, \nu(\mathrm{d}z)-bx(r)]\mathrm{d}r\Big|^2\\
&\le&\ds(t-s)\mathbb{E}\int_s^t \Big|x(r)[a-h+\sigma \int_{1}^{+\infty} z \, \nu(\mathrm{d}z)-bx(r)]\Big|^2\mathrm{d}r\\
&=&\ds(t-s)\int_s^t \mathbb{E}\Big|x(r)[a-h+\sigma \int_{1}^{+\infty} z \, \nu(\mathrm{d}z)-bx(r)]\Big|^2\mathrm{d}r\\
&\le&\ds(t-s)\int_s^t \widehat{M}\mathrm{d}r \\
&=&\ds \widehat{M}(t-s)^2.
\end{array}
\end{equation*}
Diffusion term \(I_2\):
By the It$\rm\hat{o}$ isometry,
\[
\mathbb{E}|I_2|^2 = \tau^2\int_s^t\mathbb{E}[x(r)^2]\mathrm{d}r \le \tau^2K_2(t-s).
\]
Jump term \(I_3\):
For the compensated integral, the It$\rm\hat{o}$ isometry gives
\[
\mathbb{E}|I_3|^2 = \sigma^2\int_s^t\mathbb{E}[x(r)^2]\mathrm{d}r\int_0^{+\infty} z^2\nu(\mathrm{d}z)
\le \sigma^2K_2(t-s)\int_0^{+\infty} z^{1-\beta}e^{-\lambda z}\mathrm{d}z.
\]
Since \(\beta\in(0,2)\), the integral \(\int_0^{+\infty} z^{1-\beta}e^{-\lambda z}\mathrm{d}z\) is finite.
Denote this value by \(J_\nu\). Then
\[
\mathbb{E}|I_3|^2 \le \sigma^2K_2 J_\nu\,(t-s).
\]
Let
\[
L_1 = 3\big(\tau^2K_2 + \sigma^2K_2 J_\nu\big), \qquad
L_2 = 3\widehat{M}.
\]
Then
\[
\mathbb{E}|x(t)-x(s)|^2 \le L_1(t-s) + L_2(t-s)^2.
\]
By the Cauchy-Schwarz inequality,
\[
\mathbb{E}|x(t)-x(s)| \le \sqrt{L_1(t-s) + L_2(t-s)^2}.
\]
If $t-s>1$, then $(t-s)^2>(t-s)$, we can deduce that
\[
\mathbb{E}|x(t)-x(s)| \le \sqrt{(L_1+ L_2)}(t-s).
\]
If $t-s\le1$, then $(t-s)^2\le(t-s)$, we can deduce that
\[
\mathbb{E}|x(t)-x(s)| \le \sqrt{(L_1+ L_2)}\sqrt{(t-s)}.
\]
Thus, there exists a constant \(L>0\) such that
\[
\mathbb{E}|x(t)-x(s)| \le L\big(|t-s| + \sqrt{|t-s|}\big).
\]
The same estimate holds for \(\widetilde{x}(t)\) by an identical argument. Consequently,
\begin{equation*}
\begin{array}{rcl}
\ds|\mathbb{E}|x(t) - \widetilde{x}(t)|-\mathbb{E}|x(s) - \widetilde{x}(s)|| &\le&\ds\mathbb{E}||x(t) - \widetilde{x}(t)|-|x(s) - \widetilde{x}(s)|| \\
&\le&\ds\mathbb{E}|x(t) - x(s)| + \mathbb{E}|\widetilde{x}(t) - \widetilde{x}(s)| \\
&\le&\ds 2L\big(|t-s| + \sqrt{|t-s|}\big),
\end{array}
\end{equation*}
which implies that $\mathbb{E}|x(t) - \widetilde{x}(t)|$ is uniformly continuous on $[0,+\infty)$. Combined with (\ref{w}) and applying Barbalat's lemma in \citep{a4}, we obtain
\begin{equation}\label{e1}
\lim_{t\rightarrow+\infty}\mathbb{E}\big|x(t)-\widetilde{x}(t)\big|=0.
\end{equation}
Let $P(t, x(0), Q)$ denotes the probability of $x(t)\in Q$. Lemma~3.1 and Chebyshev's inequality imply that the family $\{p(t,x(0),\cdot):t\geq 0\}$ is tight. Denote by $\mathscr{P}(\mathbb{R}_+)$ the set of all probability measures on $\mathbb{R}_+$.
For $P_1, P_2\in\mathscr{P}$ define the metric
\begin{equation*}
\mathrm{d}_U(P_{1}, P_{2}):=\sup_{f\in U}\big|\int_0^{+\infty}f(x)P_1(\mathrm{d}x)-\int_0^{+\infty}f(x)P_2(\mathrm{d}x)\big|,
\end{equation*}
where
\begin{equation*}
U=\big\{f:\mathbb{R}_+\rightarrow \mathbb{R}\big| |f(x)-f(y)|\leq \parallel x-y\parallel, |f(\cdot)|\leq 1\big\}.
\end{equation*}
For any $f\in U$ and $t,\ s>0$,
\begin{equation}\label{y}
\begin{array}{ll}
&\quad\ds\big|\mathbb{E}(f(x(t+s;x(0)))-\mathbb{E}(f(x(t;x(0))))\big|\\
&=\ds\big|\mathbb{E}\big[\mathbb{E}\big(f(x(t+s;x(0)))\big|\mathcal{F}_s\big)\big]-\mathbb{E}f(x(t;x(0)))\big|\\
&=\ds\big|\int_0^{+\infty}\mathbb{E}f(x(t;\widetilde{x}(0)))p(s,x(0),d\widetilde{x})-\mathbb{E}f(x(t;x(0)))\big|\\
&\leq\ds\int_0^{+\infty}\big|\mathbb{E}f(x(t;\widetilde{x}(0)))-\mathbb{E}f(x(t;x(0)))\big|p(s,x(0),d\widetilde{x})\\
&\leq\ds\int_{\bar{G}_K}\big|\mathbb{E}f(x(t;\widetilde{x}(0)))-\mathbb{E}f(x(t;x(0)))\big|p(s,x(0),d\widetilde{x})\\
&\quad\ds+2p(s,x(0),G^c_K),
\end{array}
\end{equation}
where $\bar{G}_K=\{x\in \mathbb{R}_+:|x|\leq K\}$ and $G^c_K=\mathbb{R}_+-\bar{G}_K$. Because $\{p(t,x(0),\cdot)\}$ is tight, we can choose $K$ sufficiently large so that $p(s,x(0),G^c_K)<\varepsilon$, $\forall s\geq 0$.

\noindent From (\ref{e1}), for any $\varepsilon>0$ there exists a $T>0$ such that
\begin{equation*}
\sup_{f\in U}\big|\mathbb{E}(f(x(t;\widetilde{x}(0)))-\mathbb{E}(f(x(t;x(0))))\big|\leq\varepsilon,\quad t\geq T.
\end{equation*}
Inserting this estimate into (\ref{y}) gives
\begin{equation*}
\big|\mathbb{E}(f(x(t+s;x(0)))-\mathbb{E}(f(x(t;x(0))))\big|\leq3\varepsilon,\quad t\geq T.
\end{equation*}
Since $f$ is arbitrary,
\begin{equation*}
\sup_{f\in U}\big|\mathbb{E}(f(x(t+s;x(0)))-\mathbb{E}(f(x(t;x(0))))\big|\leq3\varepsilon,\quad t\geq T.
\end{equation*}
Which implies that for all $t\geq T$ and $s>0$,
\begin{equation*}
\mathrm{d}_U(p(t+s,x(0),\cdot),p(t,x(0),\cdot))\leq3\varepsilon.
\end{equation*}
Thus, for each fixed $x(0)\in \mathbb{R}_+$, the family ${p(t, x(0),\cdot)}$ is Cauchy in $\mathscr{P}(\mathbb{R}_+)$.

\noindent Consequently, there exists a unique $\pi\in\mathscr{P}([0,+\infty))=\overline{\mathscr{P}(\mathbb{R}_+)}^{\rm \mathrm{d}_U}$ such that
\begin{equation}\label{z}
\lim_{t\rightarrow+\infty}\mathrm{d}_U(p(t, 1,\cdot),\pi(\cdot))=0.
\end{equation}
From (\ref{e1}) we also have
\begin{equation*}
\lim_{t\rightarrow+\infty}\mathrm{d}_U(p(t,x(0),\cdot),p(t,\psi(0),\cdot))=0,
\end{equation*}
combining  this with (\ref{z}) yields
\begin{equation*}
\begin{array}{ll}
&\ \ds\lim_{t\rightarrow+\infty}\mathrm{d}_U(p(t,x(0),\cdot),\pi(\cdot))\\
&\leq\lim \limits_{t\rightarrow+\infty}\mathrm{d}_U(p(t,x(0),\cdot),p(t,\psi(0),\cdot))
\ds+\lim_{t\rightarrow+\infty}\mathrm{d}_U(p(t,\psi(0),\cdot),\pi(\cdot))\\
&=0.
\end{array}
\end{equation*}
Therefore, the transition probability converges weakly to the unique invariant measure $\pi$, and model (\ref{f01}) is stable in distribution.
\end{proof}

\section{Optimal Harvesting}
\label{4}

In this section, we will state and prove our main results.
\medskip

\begin{theorem}
Let
\begin{equation*}
A=a - \frac{\tau^2}{2} + \int_0^{+\infty} \left[\ln(1 + \sigma z) - \sigma z \mathbf{1}_{\{0<z\leq 1\}}\right] \frac{e^{-\lambda z}}{z^{1+\beta}}\,\mathrm{d}z,
\end{equation*}
For model (\ref{f01}), if $h \in(0,A)$, then \(\Phi>0\) (which means Theorem 2.1(ii) holds) and OHE is
\begin{equation}\label{fd1}
h^*=\frac{A}{2},
\end{equation}
MESY is
\begin{equation}\label{fd2}
Y^*=\frac{A^2}{4b}.
\end{equation}
Moreover, if $h\geq A$, then $\Phi \leq 0$ and there is no optimal harvesting policy for system (\ref{f01}).

\end{theorem}
\medskip

\begin{proof}[$\mathbf{Proof\ of\ Theorem\ 4.1}$] From Theorem 3.1 we know that model (\ref{f01}) has a unique invariant measure \(\pi(\cdot)\) which is ergodic
by theorem 3.2.6 in \citep{a5}, that is to say if \(g(u)\) is an integrable function with respect to \(\pi(\cdot)\), one have
\begin{equation*}
 \lim_{t\rightarrow +\infty}\frac{1}{t} \int_0^t g(x(s))\,\mathrm{d}s = \int_0^{+\infty} g(u)\,\pi(\mathrm{d}u), \quad \text{a.s.}.
\end{equation*}
We note that
\begin{equation*}
 E(g(X(\omega)))=\int_{\Omega}g(X(\omega))\,\mathrm{d}P(\omega) = \int_0^{+\infty} g(u)\,\pi(\mathrm{d}u)
\end{equation*}
So
\begin{equation}\label{fd4}
 \lim_{t\rightarrow +\infty}E(g(x(t)))=E(g(X(\omega)))=\lim_{t\rightarrow +\infty}\frac{1}{t} \int_0^t g(x(s))\,\mathrm{d}s
\end{equation}
We now prove that the function $ g(u) = u $ is integrable with respect to the invariant measure $ \pi(\cdot) $, i.e.,
\[
\int_0^{+\infty} u\,\pi(\mathrm{d}u) < \infty.
\]
Define the truncated function $ g_N(u) := u \wedge N $ for each integer $ N > 0 $. Since $ g_N(u) \leq N $, it is bounded and hence integrable with respect to $ \pi(\cdot) $. Therefore,
\begin{equation}\label{fd5}
\int_0^{+\infty} g_N(u)\,\pi(\mathrm{d}u) = \lim_{t \to \infty} \frac{1}{t} \int_0^t g_N(x(s))\,\mathrm{d}s.
\end{equation}
Since $ g_N(x(t)) = x(t) \wedge N \leq x(t) $, by (iii) in Theorem 2.1, it follows that
\[
\lim_{t \to \infty} \frac{1}{t} \int_0^t g_N(x(s))\,\mathrm{d}s \leq \lim_{t \to \infty} \frac{1}{t} \int_0^t x(s)\,\mathrm{d}s = \frac{\Phi}{b}, \quad \text{for all } t \geq 0.
\]
Then, from (\ref{fd5}), we obtain
\begin{equation}\label{fd6}
\int_0^{+\infty} g_N(u)\,\pi(\mathrm{d}u) = \lim_{t \to \infty} \frac{1}{t} \int_0^t g_N(x(s))\,\mathrm{d}s \leq \frac{\Phi}{b}.
\end{equation}
Observe that $ g_N(u) \uparrow u $ as $ N \to +\infty $ for each $ u \in \mathbb{R}_+ $. By the Monotone Convergence Theorem, we have
\[
\int_0^{+\infty} u\,\pi(\mathrm{d}u) = \lim_{N \to \infty} \int_0^{+\infty} g_N(u)\,\pi(\mathrm{d}u) .
\]
Therefore, from (\ref{fd6}), we obtain
\[
\int_0^{+\infty} u\,\mu(\mathrm{d}u)\leq \frac{\Phi}{b} < \infty.
\]
(\ref{fd4}) together with (iii) in Theorem 2.1 show that
\begin{equation}\label{fd7}
\begin{array}{rcl}
Y(h)&=&\ds\lim\limits_{t\rightarrow+\infty}\mathbb{E}(hx(t))\\
&=&\ds h\lim\limits_{t\rightarrow+\infty}\mathbb{E}(x(t))\\
&=&\ds h\lim_{t\rightarrow +\infty}\frac{1}{t} \int_0^t x(s)\,\mathrm{d}s\\
&=&\ds\frac{h\Phi}{b}\\
&=&\ds\frac{h(a - h - \frac{\tau^2}{2} + \int_0^{+\infty} \left[\ln(1 + \sigma z) - \sigma z \mathbf{1}_{\{0<z\leq 1\}}\right] \frac{e^{-\lambda z}}{z^{1+\beta}}\,\mathrm{d}z)}{b}\\
&=&\ds\frac{h(A - h)}{b}.
\end{array}
\end{equation}
It is easy to know that $Y(h)$ attains its unique extreme value at
\[
h^*=\frac{A}{2}.
\]
and MESY (\ref{fd2}) follows directly.

Furthermore, if $\Phi \leq 0$, no effective optimal harvesting policy exists. When $\Phi = 0$, we have $Y(h) = 0$ for all $h \geq 0$ due to (\ref{fd7}) and Theorem~2.1(ii); When $\Phi < 0$, the population goes extinct almost surely by Theorem~2.1(i), leading to zero yield. Thus, optimal harvesting is only well-defined when $\Phi > 0$.
\end{proof}
\medskip

\rm
\section{Numerical Simulations}
\label{6}

In this section, we conduct numerical simulations to validate the theoretical results derived in the previous sections and to provide a quantitative basis for subsequent analysis. The stochastic differential equation (\ref{f01}) is discretized using the Euler--Maruyama scheme~\citep{a6}, and all computations are implemented in Python. The time step is set to $\Delta = 0.001$, and the total simulation horizon is $T = 2000$ to ensure that the system reaches its stationary regime.

Unless otherwise stated, the baseline parameter values are chosen as
\[
a = 1.5,\quad b = 0.1,\quad \beta = 0.7,\quad \lambda = 1.0,\quad
\tau^2 = 0.01,\quad \sigma = 0.005.
\]
The initial condition is set to $x(0) = 1.0$.

\subsection*{(1) Analysis of the Impact of Control Variable $h$}

We first verify the existence of an optimal harvesting strategy and examine the impact of harvesting effort on population dynamics. Based on equations (\ref{fd1}) and (\ref{fd2}), under the baseline parameter setting, OHE is computed as $h^* = 0.75$, corresponding to MESY $Y^* = 5.60$.

To illustrate the population trajectories and sustainable yields under different harvesting intensities, four representative harvesting effort levels are selected for simulation, and the results are summarized in Table~1. The sign of the key parameter $\Phi$ determines the long-term behavior of the system (Theorem~2.1).

\begin{table}[h!]
\centering
\caption{System response under different harvesting efforts (baseline parameters).}
\small
\begin{tabular}{|c|c|c|c|c|c|}
\hline
 $h$ & vs. $h^*$ & $\Phi$ & Theoretical Behavior & $\frac{1}{T}\int_0^T x(t)dt$ &$Y(h)$ \\
\hline
0.22 & $< h^*$ & 1.27 & Persistence in the mean & $12.72$ & $2.85$ \\
\hline
0.75 & $= h^*$ & 0.75 & Persistence in the mean & $7.48$ & $5.60$ \\
\hline
1.12 & $> h^*$ & 0.37 & Persistence in the mean & $3.74$ & $4.20$ \\
\hline
1.65 & $> h^*$ & -0.15 & Extinction & (tends to 0) & $0.00$ \\
\hline
\end{tabular}
\end{table}

\begin{figure}[htbp]
  \centering
  \footnotesize
  \includegraphics[width=0.8\textwidth]{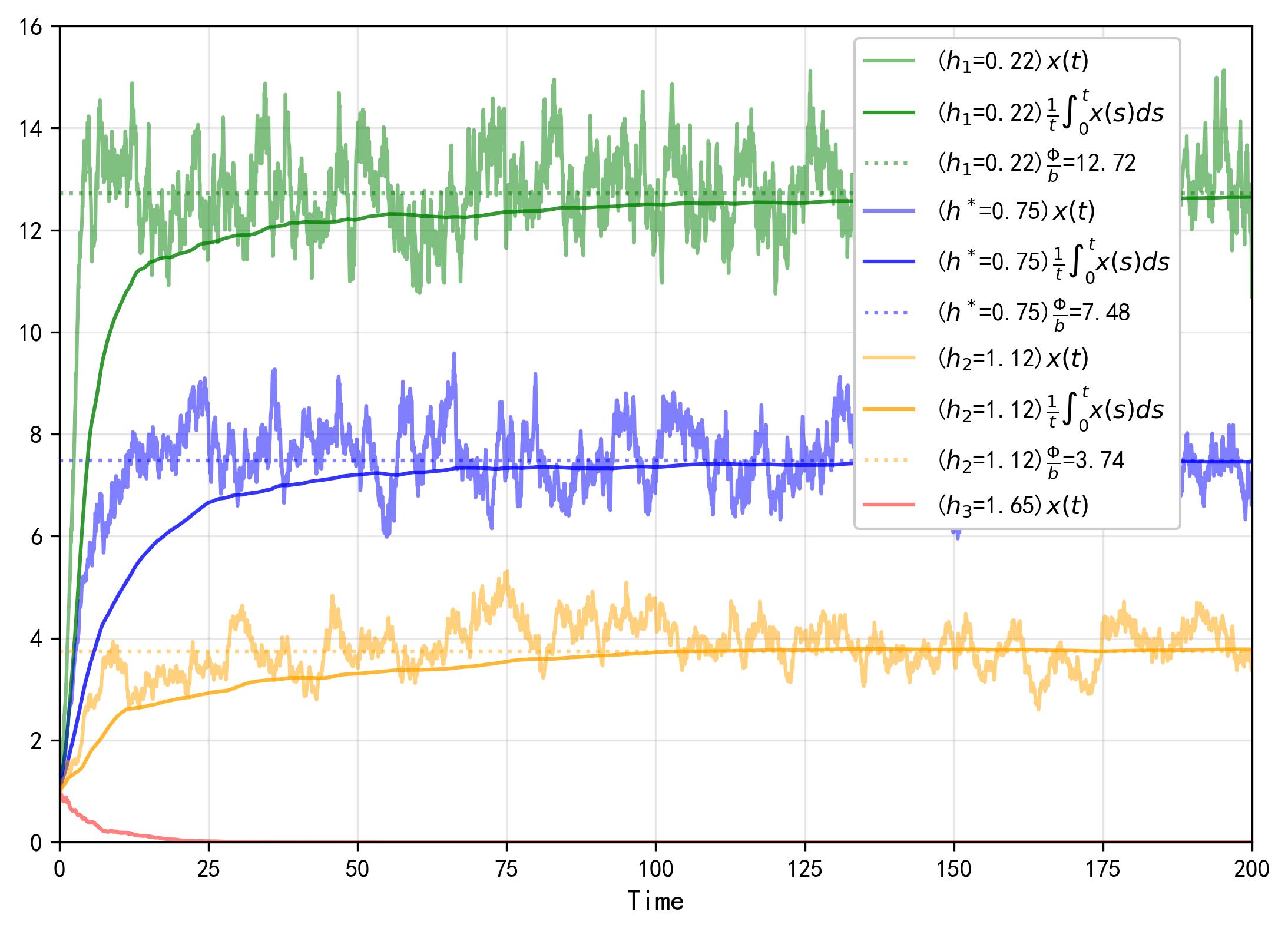}
  \caption{\footnotesize Population dynamics under different harvesting efforts.}
  \label{fig:FIG1}
\end{figure}

\begin{figure}[htbp]
  \centering
  \includegraphics[width=0.8\textwidth]{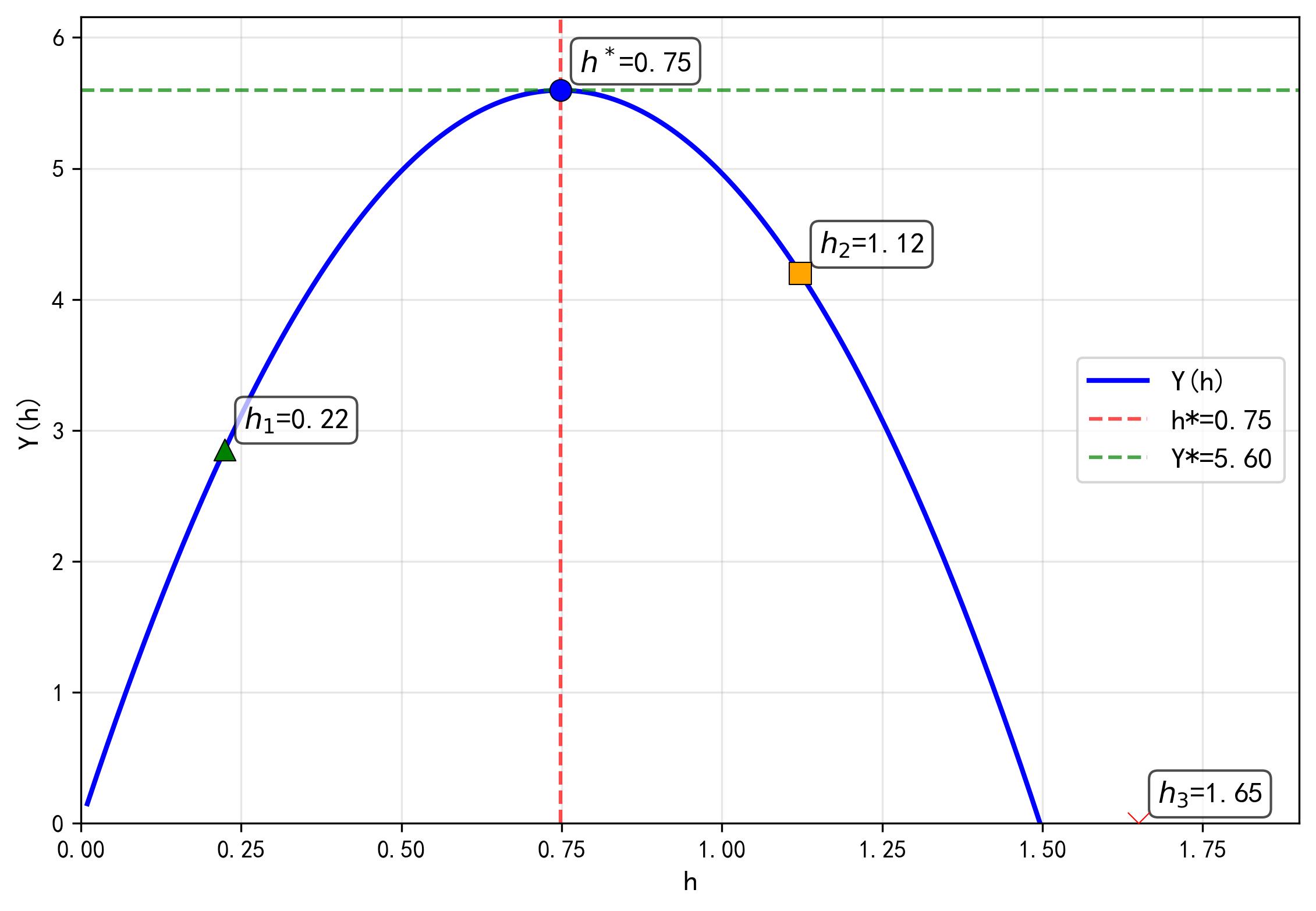}
  \caption{\footnotesize Optimal harvesting policy $Y(H)$ as a function of $h$.}
  \label{fig:FIG2}
\end{figure}

\begin{itemize}
    \item \textbf{\cref{fig:FIG1}} displays typical sample paths of the population process $x(t)$ under the four scenarios listed in Table~1. As shown, when $\Phi > 0$, the population fluctuates around a positive equilibrium level under stochastic perturbations; when $\Phi < 0$ (e.g., $h = 1.65$), the population declines to extinction within finite time. These simulation results are fully consistent with the theoretical predictions of Theorem~2.1.

    \item \textbf{\cref{fig:FIG2}} plots the theoretical ESY function $Y(h)$ over the range $h \in [0, 1.9]$, as given in Theorem~4.1. The theoretical optimal point $(h^*, Y^*) = (0.75, 5.60)$ is explicitly marked. For validation, ESY corresponding to the four representative harvesting efforts in Table~1 are superimposed as scatter points. The simulation results closely match the theoretical curve, providing strong numerical support for equations (\ref{fd1}) and (\ref{fd2}). Moreover, the figure clearly demonstrates the concavity of $Y(h)$ with respect to $h$: the yield increases with harvesting effort initially, reaches its maximum at $h^*$, and then decreases thereafter.
\end{itemize}

\subsection*{(2) Sensitivity Analysis with Respect to Noise Intensities $\tau^2$ and $\sigma$}

This subsection investigates the influence of environmental noise intensities on the optimal management strategy. Keeping the structural parameters fixed at $(\beta, \lambda) = (0.7, 1.0)$, we separately examine the effects of continuous noise intensity $\tau^2$ and jump noise intensity $\sigma$.

\begin{figure}[htbp]
  \centering
  \includegraphics[width=0.9\textwidth]{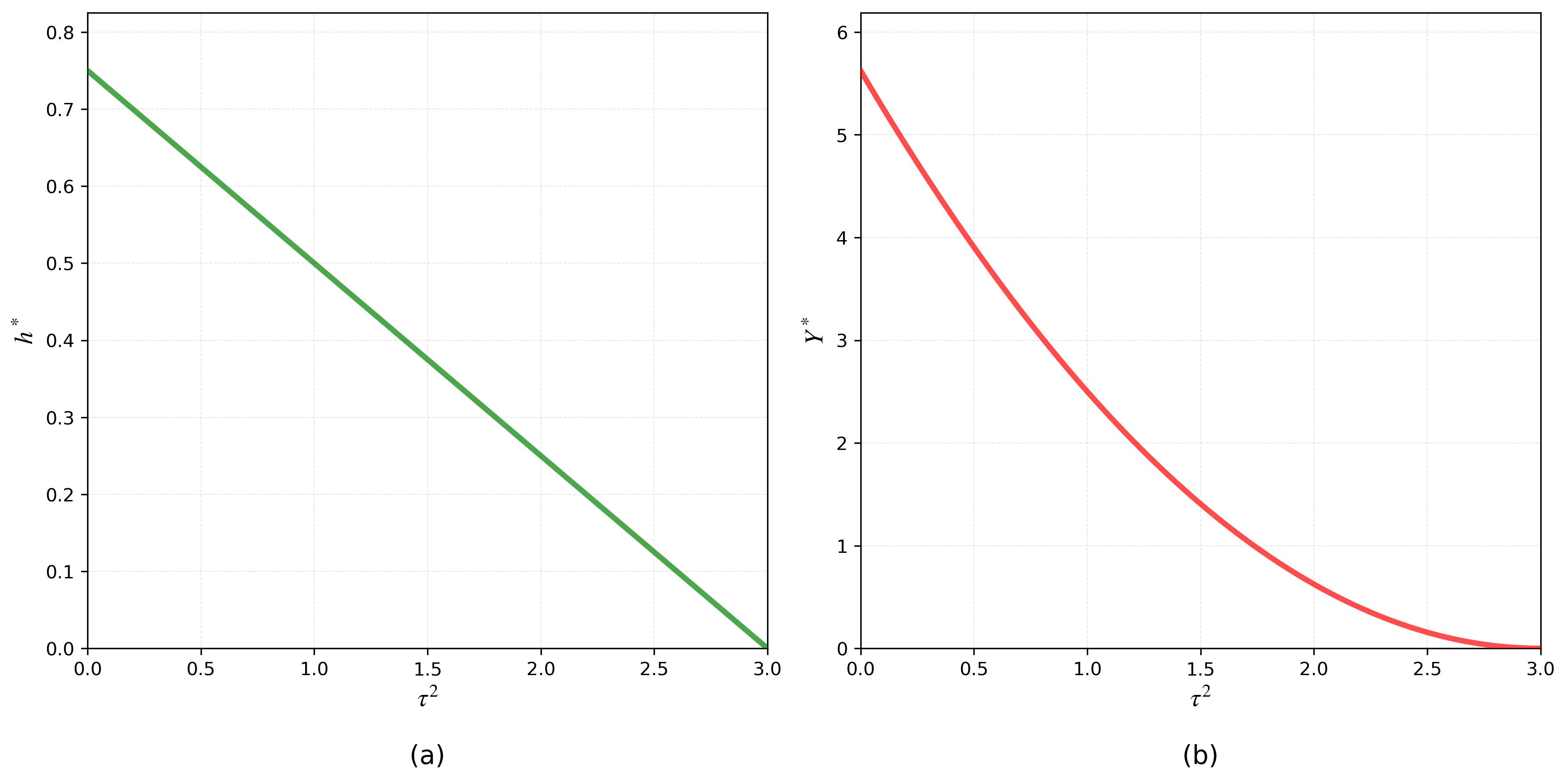}
  \caption{\footnotesize Dependence of OHE $h^*$ and MESY $Y^*$ on white noise intensity $\tau^2$.}
  \label{fig:FIG3}
\end{figure}

\begin{figure}[htbp]
  \centering
  \includegraphics[width=0.9\textwidth]{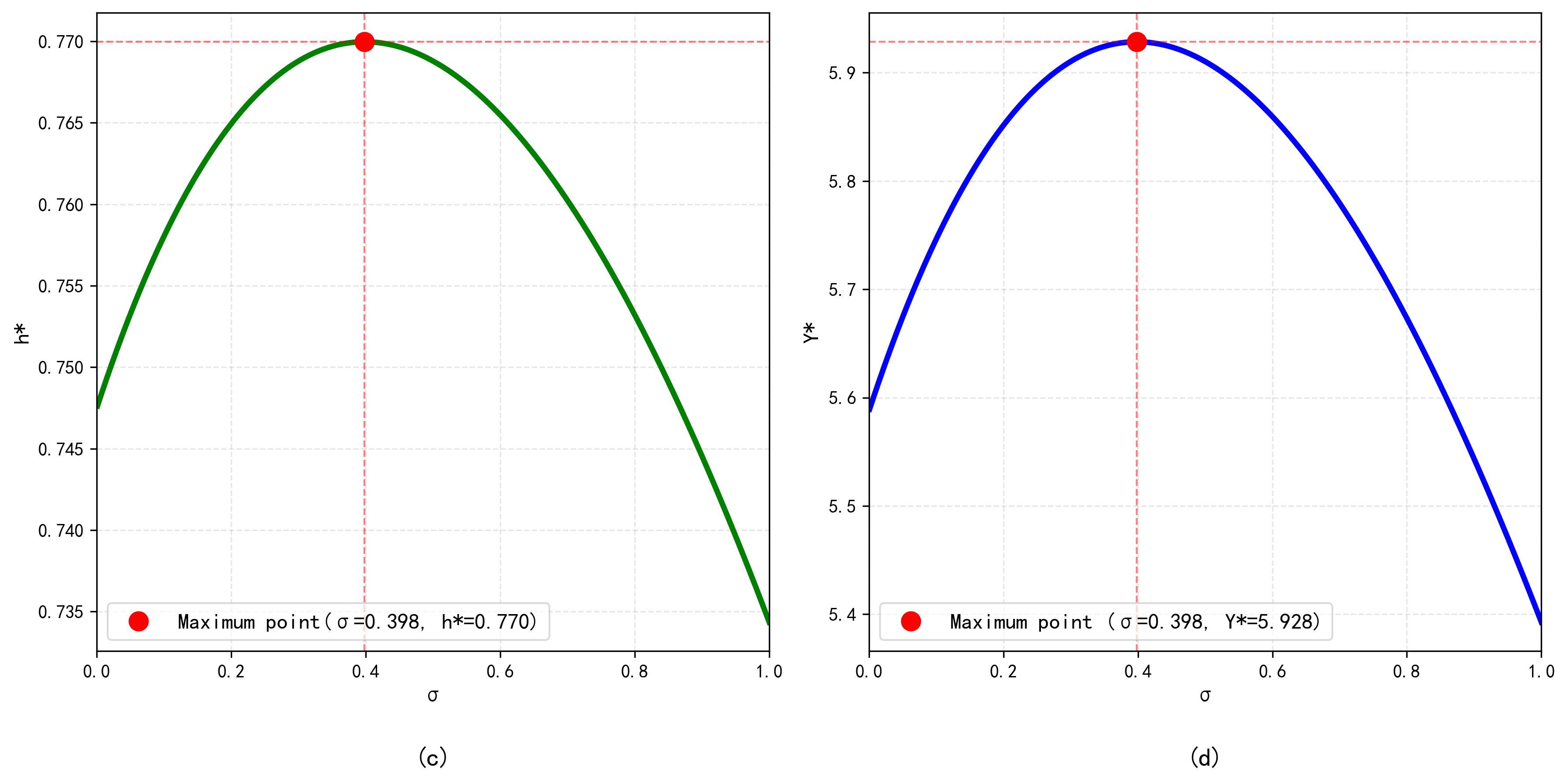}
  \caption{\footnotesize Dependence of OHE $h^*$ and MESY $Y^*$ on jump noise intensity $\sigma$.}
  \label{fig:FIG4}
\end{figure}

\begin{itemize}
    \item \textbf{\textbf{\cref{fig:FIG3}}} illustrates the impact of continuous noise. Subfigure~(a) shows the variation of OHE $h^*$ as a function of $\tau^2$, while subfigure~(b) presents the corresponding MESY $Y^*$. In these simulations, $\sigma$ is fixed at $0.005$.

    \item \textbf{\textbf{\cref{fig:FIG4}}} depicts the influence of jump noise. Subfigure~(c) shows how $h^*$ varies with $\sigma$, and subfigure~(d) shows the associated changes in $Y^*$. In this case, $\tau^2$ is fixed at $0.01$.
\end{itemize}

\subsection*{(3) Sensitivity Analysis with Respect to Structural Parameters $(\beta, \lambda)$}

This part constitutes the core of the present study and aims to reveal the complex effects of intrinsic structural characteristics, captured by the parameters $\beta$ and $\lambda$, on the optimal management strategy. With noise intensities fixed at their baseline values $(\tau^2 = 0.01, \sigma = 0.005)$, we perform a systematic scan over the parameter domain
\[
\Omega = \left\{ (\beta, \lambda) \mid \beta \in (0, 2.0),\ \lambda \in (0, 4.0) \right\}.
\]

\begin{figure}[htbp]
  \centering
  \includegraphics[width=0.7\textwidth]{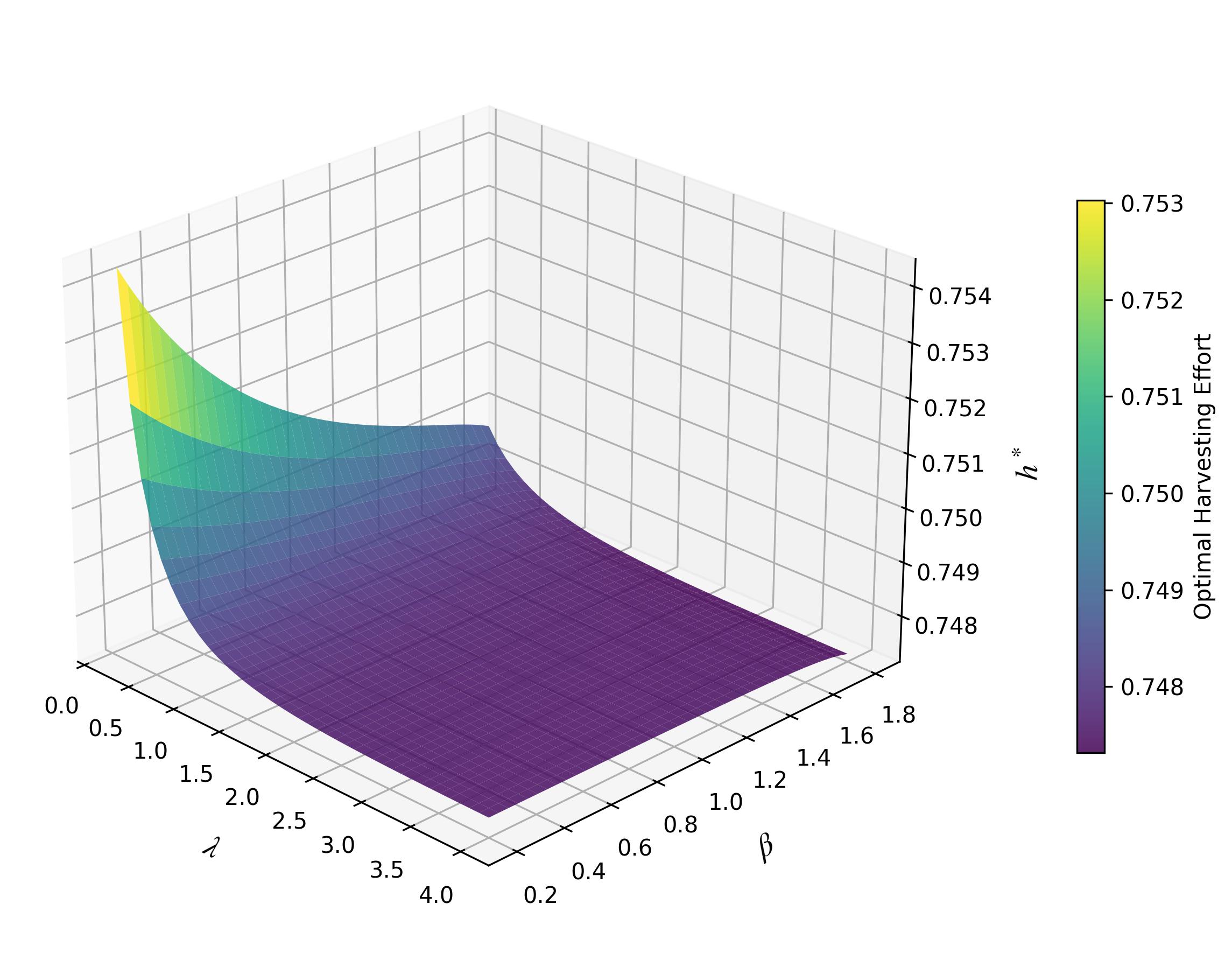}
  \caption{\footnotesize Dependence of OHE $h^*$ on ($\beta$,$\lambda$).}
  \label{fig:FIG5}
\end{figure}

\begin{figure}[htbp]
  \centering
  \includegraphics[width=0.7\textwidth]{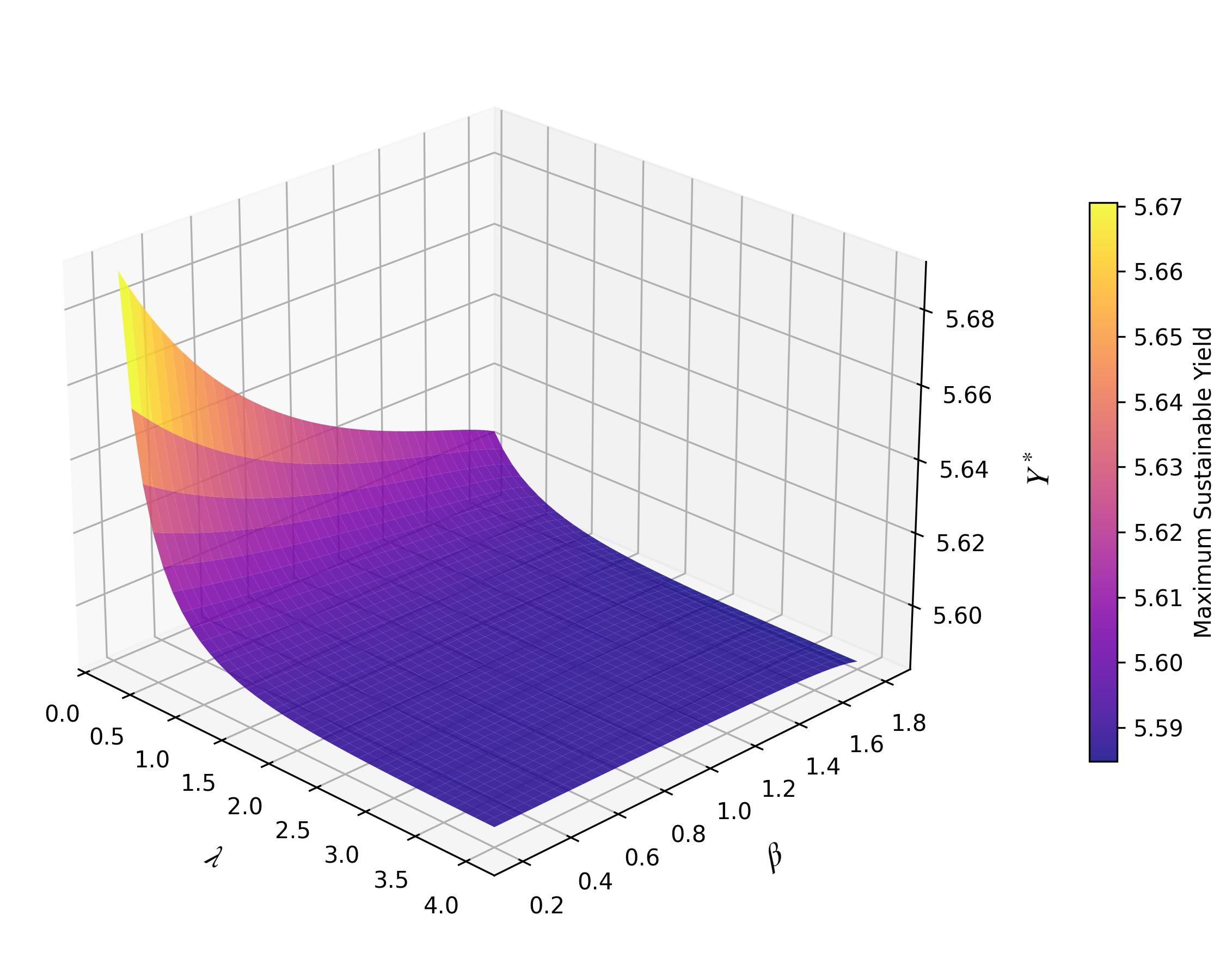}
  \caption{\footnotesize Dependence of MESY $Y^*$ on ($\beta$,$\lambda$).}
  \label{fig:FIG6}
\end{figure}

\begin{itemize}
    \item \textbf{\textbf{\cref{fig:FIG5}}} presents a heat map of OHE $h^*$ over the $(\beta, \lambda)$ parameter plane. The figure clearly reveals the nonlinear dependence of $h^*$ on the structural parameters. The color gradient ranges from purple (low values) to yellow (high values), indicating the magnitude of $h^*$.

    \item \textbf{\textbf{\cref{fig:FIG6}}} shows the distribution of MESY $Y^*$ over the same $(\beta, \lambda)$ plane using an identical visualization scheme. By comparing this with \textbf{\cref{fig:FIG5}}, one can directly contrast the effects of structural parameters on OHE $h^*$ and MESY $Y^*$.
\end{itemize}

Overall, the numerical results presented in this section provide a solid empirical foundation for the analysis and discussion in the next section. The various dependencies and patterns revealed in Figures~1--6 will be examined in depth in Section~6, from which clear management implications will be distilled.

\medskip

\rm
\section{Analysis and Conclusions}
\label{7}

In this paper, we formulate a stochastic Logistic model driven jointly by Brownian motion and a one-sided truncated stable L\'evy process, and systematically investigate the optimal harvesting problem for renewable resources under a composite noise environment. We establish threshold conditions for population persistence and extinction (Theorem~2.1), derive explicit analytical expressions for the optimal harvesting effort $h^*$ and MESY $Y^*$ (Theorem~4.1), and employ numerical simulations (\cref{fig:FIG1,fig:FIG2,fig:FIG3,fig:FIG4,fig:FIG5,fig:FIG6}) to reveal the following key findings.

\subsection*{(1) Dual Effects of Harvesting Effort $h$ and a Management Decision Framework}

Harvesting effort $h$ exerts a dual influence on population dynamics:
\begin{itemize}
  \item \textbf{Production effect}: Moderate harvesting enhances resource utilization by regulating population density and alleviating intraspecific competition, thereby increasing per-capita productivity and harvest yield.
  \item \textbf{Extinction effect}: Excessive harvesting directly removes too many individuals, reduces the population base, weakens reproductive potential, and may push the population toward extinction.
\end{itemize}
The balance between these opposing forces determines the properties of ESY $Y(h)$.

Theoretical analysis shows that a key parameter $\Phi$ governs the long-term fate of the population (\cref{fig:FIG1}): when $\Phi > 0$, the population persists in mean; when $\Phi < 0$, the population tends toward extinction. Within this safety boundary, ESY $Y(h)$ is a concave function of $h$, admitting a unique maximizer $h^*$ that achieves the maximum yield (\cref{fig:FIG2}).

These results provide a clear management framework: decision-makers must strike a balance between ecological constraints, keeping $h$ strictly below the critical threshold to ensure long-term population persistence, and economic optimality, adjusting $h$ toward $h^*$ to maximize sustainable yield.

\subsection*{(2) Differential Impacts and Critical Phenomena Induced by Environmental Noises $\tau^2$ and $\sigma$}

The effects of continuous environmental noise with intensity $\tau^2$ and jump noise with intensity $\sigma$ on the optimal harvesting strategy exhibit fundamentally different mechanisms.

\begin{itemize}
  \item \textbf{Effect of $\boldsymbol{\tau^2}$.} The continuous noise intensity $\tau^2$ exerts a strictly monotonic negative impact on both OHE $h^*$ and MESY $Y^*$ (see \cref{fig:FIG3}):
      \begin{equation*}
      \frac{\partial h^*}{\partial (\tau^2)} = -\frac{1}{4} < 0, \qquad
      \frac{\partial Y^*}{\partial (\tau^2)} = -\frac{h^*}{4b} < 0.
      \end{equation*}
      This result indicates that increasing environmental variability reduces both the optimal harvesting intensity and the achievable sustainable yield, implying that more conservative harvesting policies are required under stronger continuous fluctuations.
  \item \textbf{Effect of $\boldsymbol{\sigma}$.} In contrast, the influence of jump noise intensity $\sigma$ on $h^*$ and $Y^*$ is non-monotonic (see \cref{fig:FIG4}):
      \begin{equation*}
      \frac{\partial h^*}{\partial \sigma} = \frac{1}{2} I'(\sigma), \qquad
      \frac{\partial Y^*}{\partial \sigma} = \frac{2h^*}{b}\frac{\partial h^*}{\partial \sigma},
      \end{equation*}
      where
      \begin{equation*}
      I'(\sigma)
      = \int_0^\infty
      \left[
      \frac{1}{1+\sigma z}
      - \mathbb{I}_{\{0<z\le 1\}}
      \right]
      \frac{e^{-\lambda z}}{z^{\beta}} \, \mathrm{d}z .
      \end{equation*}

      The analysis reveals the existence of a unique critical value $\sigma_0 > 0$ such that
      \begin{equation*}
      \frac{\partial h^*}{\partial \sigma}
      \begin{cases}
      > 0, & \sigma < \sigma_0, \\[4pt]
      = 0, & \sigma = \sigma_0, \\[4pt]
      < 0, & \sigma > \sigma_0 .
      \end{cases}
      \end{equation*}
      Consequently, moderate jump disturbances ($\sigma < \sigma_0$) enhance both OHE and MESY, whereas excessive jump intensity ($\sigma > \sigma_0$) suppresses them simultaneously.
\end{itemize}

This finding uncovers a non-monotonic regulatory effect of environmental perturbations on sustainable population yield. To our knowledge, it provides the first rigorous mathematical characterization and computable critical threshold for such phenomena within a stochastic optimal harvesting framework, offering new quantitative insights into the role of environmental shocks in ecological resource management.

\subsection*{(3) Effects and mechanisms of intervention structure parameters $(\beta,\lambda)$}

The statistical structure parameters $(\beta,\lambda)$ governing management interventions play a fundamental role in determining system potential. Numerical results (\cref{fig:FIG5,fig:FIG6}) indicate that both parameters significantly influence the optimal harvesting level $h^*$ and the sustainable yield $Y^*$.

\begin{itemize}
    \item \textbf{Effect of $\boldsymbol{\lambda}$.} Increasing $\lambda$ significantly reduces both $h^*$ and $Y^*$. This is because a larger $\lambda$ imposes a stricter upper bound on individual intervention sizes, limiting the system's ability to respond to sudden disturbances and thereby weakening its recovery capacity. As a result, harvesting strategies must become more conservative.
    \item \textbf{Effect of $\boldsymbol{\beta}$.} Increasing $\beta$ also leads to a monotonic decline in $h^*$ and $Y^*$. From a mechanistic perspective, increasing \( \beta \) shifts the L\'evy measure toward small jumps. Because the L\'evy measure assigns more weight to small jumps as \(\beta\) increases, and the integrand \( \ln(1 + \sigma z) - \sigma z\) is negative for small \(z\), the integral becomes more negative, thereby reducing \(\Phi\). From a biological perspective, increasing \( \beta \) corresponds to more frequent but smaller-scale interventions. Every increase in density intensifies the competitive term \( -bx^2 \). Compared with low-frequency large-scale interventions, high-frequency small-scale interventions yield lower marginal benefit, while the competitive cost accumulates continuously due to frequent interventions. Overall, random and frequent small interventions reduce both optimal harvesting levels and long-term population performance.
\end{itemize}

It should be emphasized that the jump term represents stochastic interventions rather than a deterministic increase in population size. Although all jumps are positive, their net effect on population growth is not necessarily positive due to the combined influence of stochasticity and nonlinear dynamics. In particular, frequent small-scale interventions may fail to accumulate effectively and may even reduce the overall growth capacity of the system. This theoretical inference is supported by empirical evidence from fishery stock enhancement: Taylor et al. \citep{a7} demonstrated that repeated high-frequency releases of juvenile mulloway at the same site (analogous to large \(\beta\)) led to faster emigration, greater use of sub-optimal habitats, and higher cumulative predation risk compared to single, well-dispersed releases, ultimately resulting in greater losses from the stocked system.

Moreover, significant interaction effects exist between $\beta$ and $\lambda$, implying that these parameters must be designed in a coordinated manner. Table 2 summarizes intervention paradigms under different $(\beta,\lambda)$ combinations and their corresponding management recommendations.

\begin{table}[h!]
\centering
\caption{Intervention management strategies based on $(\beta,\lambda)$ parameters.}
\small
\setlength{\tabcolsep}{3pt}
\begin{tabular}{|l|l|l|l|}
\hline
\makecell{Parameter\\Combination} &
\makecell{Management\\Paradigm} &
\makecell{Harvesting\\Strategy} &
\makecell{Typical\\Scenario} \\
\hline
\makecell{Small $\beta$\\Small $\lambda$} &
\makecell{Low-frequency,\\ relatively large-scale\\ + \\Large-scale emergency-ready} &
\makecell{Relatively high\\($h^*$ larger)} &
\makecell{Chinese sturgeon: \\breeding-season stocking + \\emergency fund for \\pollution events} \\
\hline
\makecell{Small $\beta$\\Large $\lambda$} &
\makecell{Low-frequency,\\ relatively large-scale\\ +\\Scale-constrained} &
\makecell{Moderately\\conservative\\($h^*$ medium)} &
\makecell{Qinghai Lake naked carp:\\ breeding-season stocking, \\limited hatchery capacity} \\
\hline
\makecell{Large $\beta$\\Small $\lambda$} &
\makecell{High-frequency, small-scale\\ +\\Large-scale emergency-ready} &
\makecell{Conservative,\\dynamic\\($h^*$ lower)} &
\makecell{Urban rivers: weekly \\stocking + emergency\\ reserve for pollution events} \\
\hline
\makecell{Large $\beta$\\Large $\lambda$} &
\makecell{High-frequency, small-scale\\+\\Scale-constrained} &
\makecell{Highly\\conservative\\($h^*$ lowest)} &
\makecell{Reservoirs: monthly\\ fixed-quota stocking of\\ filter-feeding fish\\ for algae control} \\
\hline
\end{tabular}
\end{table}

In summary, optimal resource management requires coherence across three levels: ensuring harvesting effort does not exceed the survival threshold ($\Phi > 0$), using $h^*$ as an operational benchmark; adjusting strategies in response to environmental noise intensities ($\tau^2$ and $\sigma$), with jump disturbances tuned toward the optimal level $\sigma_0$; and optimizing intervention structures $(\beta, \lambda)$ to enhance system resilience. This reflects a paradigm shift from passive adaptation to environmental uncertainty toward active system design.

Future research may extend this framework in several directions. First, the single-species model can be generalized to multi-species systems to examine how trophic cascades, interspecific competition, and mutualistic interactions affect optimal management strategies. Second, calibrating and validating the model using empirical data from specific ecosystems, such as fisheries, forests, or protected areas, would substantially enhance predictive accuracy and credibility. Finally, developing decision-support tools or case-based management libraries for practitioners would facilitate the translation of theoretical results into adaptive management strategies applicable in complex and stochastic environments. Together, these efforts will help bridge theory and practice and provide robust scientific foundations for achieving ecological and economic sustainability under uncertainty.
\medskip



\rm

\end{document}